\RequirePackage{fix-cm}
\documentclass[smallextended]{svjour3}       
\smartqed  
\usepackage{graphicx}
\usepackage{hyperref}
\usepackage[latin9]{inputenc}
\usepackage{mathrsfs}
\usepackage{amsmath}
\usepackage{amssymb}
\usepackage{esint}
\usepackage{enumerate}
\usepackage{bbm}
\usepackage[usenames]{color}
\usepackage{algorithm,algorithmic}  
\usepackage{listings}

\newcommand{\N}{\mathbb{N}}
\newcommand{\R}{\mathbb{R}}

\newcommand{\LL}{\mathcal{L}}
\newcommand{\BB}{\mathcal{B}}
\newcommand{\FF}{\mathcal{F}}
\newcommand{\sS}{\mathbb{S}}
\newcommand{\1}{\ensuremath{\mathbbm{1}}}

\newcommand{\dx}[1][x]{\ensuremath{\,{\rm{d}} #1}}

\def\norm#1{\hspace{0.2ex} \|#1\| \hspace{0.2ex}} 
 
\newcommand{\kommentar}[1]{}

 \def\makeheadbox{\centering This is a post-peer-review, pre-copyedit version of an article
published in \emph{Jahresber.\ Dtsch.\ Math.\ Ver.} The final authenticated version is available online at:
\url{https://doi.org/10.1365/s13291-021-00236-2}.}

\begin{document}

\title{An introduction to finite element methods for inverse coefficient problems in elliptic PDEs}

\titlerunning{Introduction to FEM for inverse coefficient problems in elliptic PDEs}      

\author{Bastian Harrach}

\institute{B. Harrach \at 
              Institute for Mathematics, Goethe-University Frankfurt, Frankfurt am Main, Germany \\
              \email{harrach@math.uni-frankfurt.de} 
}

\date{}

\maketitle

\begin{abstract}
Several novel imaging and non-destructive testing technologies are based on reconstructing the spatially dependent coefficient in an elliptic partial differential equation
from measurements of its solution(s). In practical applications, the unknown coefficient is often assumed to be piecewise constant on a given pixel partition (corresponding to the 
desired resolution), and only finitely many measurement can be made. This leads to the problem of inverting a finite-dimensional non-linear forward operator $\FF:\ \mathcal D(\FF)\subseteq \R^n\to \R^m$, where evaluating $\FF$ requires one or several PDE solutions. 

Numerical inversion methods require the implementation of this forward operator and its Jacobian. We show how to efficiently implement both using a standard FEM package and prove convergence of the FEM approximations against their true-solution counterparts. We present simple example codes for Comsol with the Matlab Livelink package, and numerically demonstrate the challenges that arise from 
non-uniqueness, non-linearity and instability issues. We also discuss monotonicity and convexity properties of the forward operator that arise for symmetric measurement settings. 

This text assumes the reader to have a basic knowledge on Finite Element Methods, including the variational formulation of elliptic PDEs, 
the Lax-Milgram-theorem, and the C\'ea-Lemma. Section \ref{sect:true_solution} also assumes that the reader is familiar with the concept of Fr\'echet differentiability.
\end{abstract}

\section{Introduction}

Many practical reconstruction problems in the field of medical imaging and non-destructive testing lead to inverse coefficient problems in elliptic partial differential equations.
This text is meant to be an introductory tutorial for implementing such problems with Finite Element Methods (FEM).

We assume that the unknown coefficient is piecewise-constant on a given resolution, and that finitely many linear measurements of one of several solutions 
are taken, where different solutions are generated by different linear excitation in the underlying physics model. 
This leads to the finite-dimensional non-linear inverse problem of
determining
\[
\sigma\in \R^n \quad \text{ from } \quad \FF(\sigma)\in \R^m
\]
with $n\in \N$ unknowns and $m\in \N$ measurements. 

Iterative numerical solution methods for this inverse problem require evaluating $\FF$ and its derivatives at each iteration step,
which means solving the underlying elliptic PDE. In this work, we will demonstrate how FEM-based implementations for $\FF$ and its Jacobian can be obtained very efficiently from standard FEM-solvers for the considered elliptic PDE. Roughly speaking, the sensitivity of a measurement with respect to changing the coefficient in one pixel can be simply calculated
by multiplying FEM-solutions corresponding to the measurement and excitation patterns with so-called pixel stiffness matrices that are obtained from
summing up all element stiffness matrices of elements belonging to the pixel where the change occurs.
Hence, the FEM-based Jacobian can be obtained without any additional computational burden with just a few lines of extra code.
Alternatively, for an even simpler implementation, the pixel stiffness matrices can be easily obtained by subtracting
global stiffness matrices without requiring any knowledge about the triangulation details.

This text is meant as a simple-to-read explanation of this approach in a sufficiently general but naturally arising setting. 
More precisely, we restrict ourselves to coercive and symmetric variational formulations that linearly 
depend on the unknown coefficients, and to excitations and measurements that correspond to linear functionals. 
In this setting, we demonstrate how to obtain the Jacobian 
of the FEM-based forward map with the means of a standard FEM software package such as COMSOL.
We also discuss monotonicity and convexity properties arising
in symmetric measurement situations that are the basis for recent research on rigorously justified reconstruction methods.

The purpose of this text is of introductory nature, but we proceed in a mathematically rigorous fashion to allow this text to also
serve as a reference. We prove differentiability of the true-solution forward operator and its FEM-based approximation, 
and show convergence of the FEM-approximated quantities to their true-solution counterparts.

Section~\ref{sec:Examples} gives two examples to motivate our general setting: stationary diffusion and Elecrical Impedance Tomography. Section~\ref{sect:true_solution} introduces the forward operator
using the exact PDE solution and derives its properties. The FEM-approximation of the forward operator and its Jacobian is studied in section~\ref{sect:FEM_setting}. 
In section~\ref{sect:numerics} we show numerical examples and demonstrate some of the major challenges that arise in solving inverse coefficient problems.
The COMSOL/MATLAB source codes for all examples are given in appendix~\ref{sect:appendix}.

\section{Motivation and examples}\label{sec:Examples}

\subsection{Stationary diffusion}\label{subsect:diffusion}

We consider the stationary diffusion equation 
\begin{equation}\label{eq:diffusion}
-\nabla \cdot (\sigma \nabla u)=g \quad \text{ in $\Omega$}
\end{equation}
with homogeneous Dirichlet boundary condition $u|_{\partial \Omega}=0$ in a Lipschitz bounded domain $\Omega\subset \R^d$, $d\in \N$.
For $u\in H^1(\Omega)$, $\sigma\in L^\infty_+(\Omega)$ and $g\in L^2(\Omega)$ the equation is equivalent to finding $u\in H_0^1(\Omega)$ with
\begin{equation}\label{eq:diffusion_VarForm}
\int_\Omega \sigma \nabla u \cdot \nabla v\dx = \int_\Omega g v \dx \quad \text{ for all } v\in H_0^1(\Omega),
\end{equation}
and unique solvability follows from the Lax-Milgram theorem.

We are interested in the inverse coefficient problem of determining the diffusivity coefficient $\sigma$ in \eqref{eq:diffusion} from measurements of the solution 
for one or several source terms $g$, cf.\ \cite{hanke1997regularizing} for an application in groundwater filtration.
In practical applications with finitely many measurements, it is natural to only aim for a certain pixel-based resolution 
and therefore assume that $\sigma$ is piecewise constant with respect to a partition $\Omega=\bigcup_{i=1}^n \mathcal  P_i$, i.e.
\[
\sigma(x)=\sum_{i=1}^n \sigma_i \chi_{\mathcal P_i}(x) \quad \text{ for all } x\in \Omega,
\]
where the pixels $\mathcal P_i\subseteq \Omega$ are assumed to be measurable subsets. The left image in figure \ref{fig:diffusion_example} shows a simple example where
the unit square $\Omega=(0,1)^2$ is divided into $3\times 3$ pixels. In the following, with a slight abuse of notation, we write $\sigma=(\sigma_1,\ldots,\sigma_n)^T\in \R^n$ for the unknown diffusivity.

\begin{figure*}
\begin{tabular}{c c}
  \includegraphics[width=0.45\textwidth]{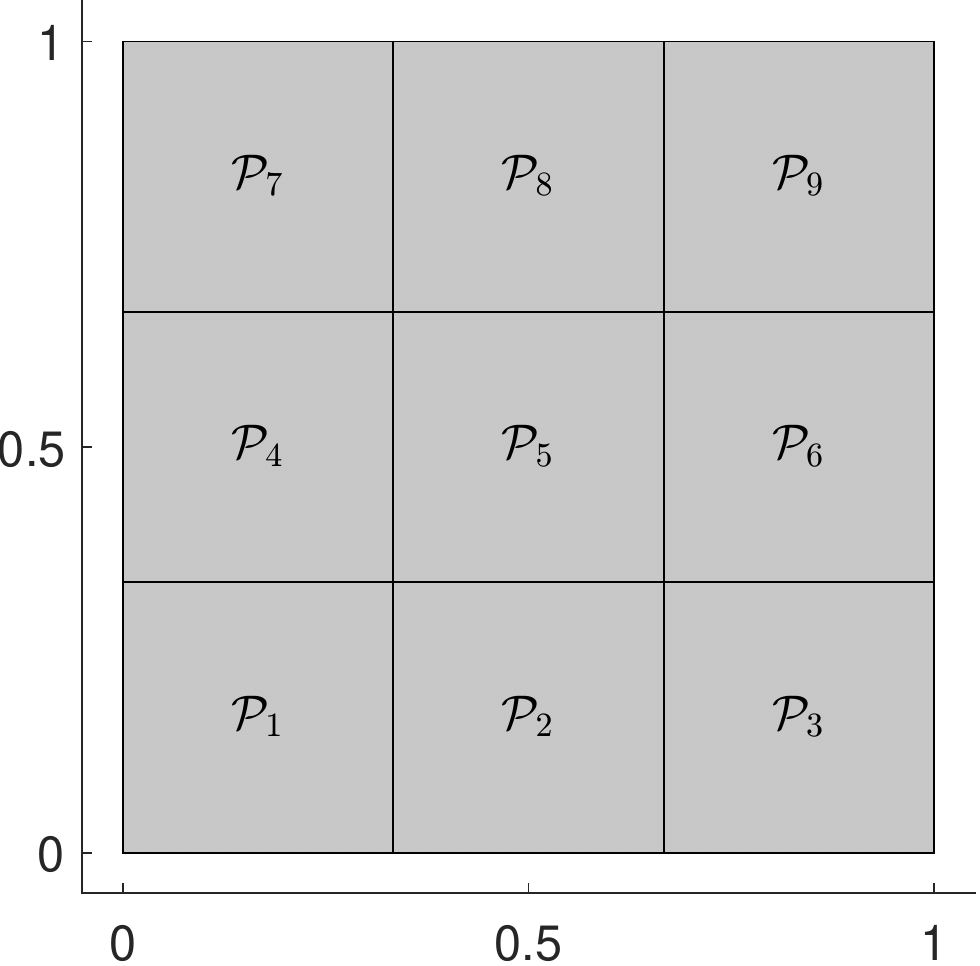} &
  \includegraphics[width=0.45\textwidth]{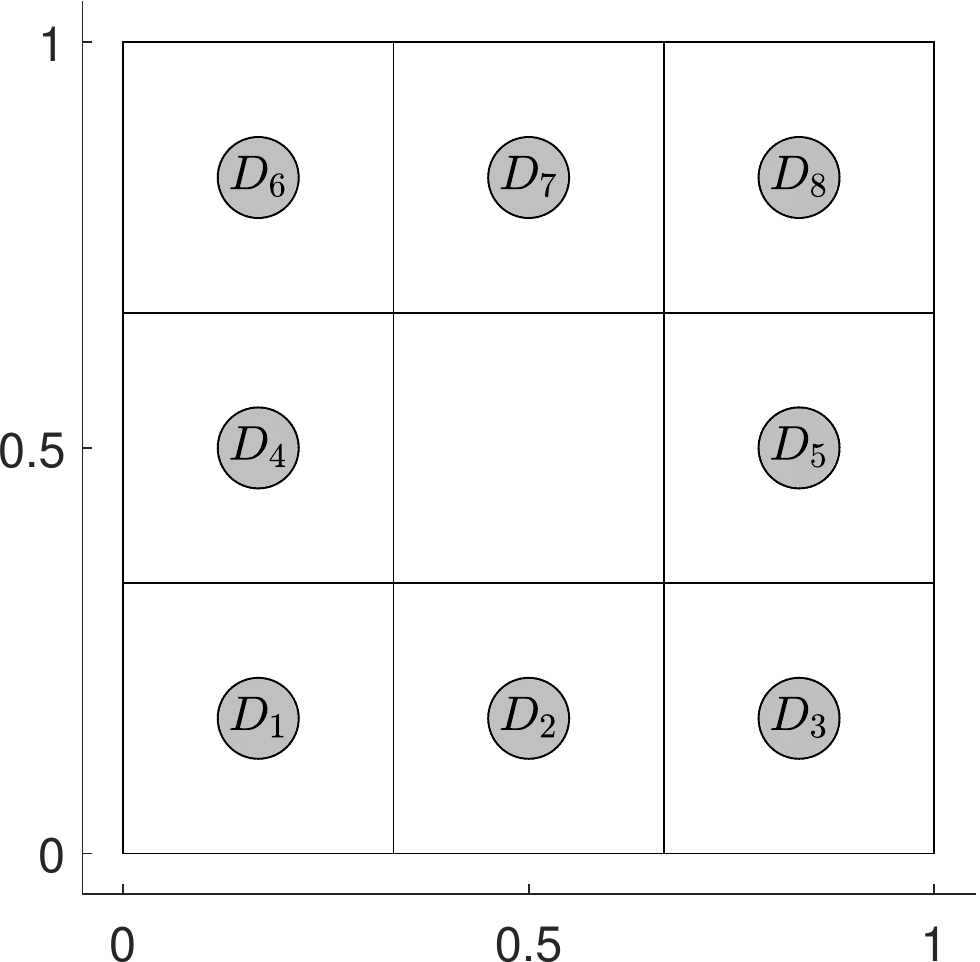} 
\end{tabular}
\caption{Pixel partition and circular subdomains used for excitations and measurements.}
\label{fig:diffusion_example}       
\end{figure*}

The source term $g$ in the diffusion model \eqref{eq:diffusion} can be identified with the linear functional
on the right hand side of the variational formulation \eqref{eq:diffusion_VarForm}
\[
l\in H^{-1}(\Omega),\quad l(v):= \int_\Omega g v \dx,
\]
which corresponds to identifying $L^2(\Omega)$ with a subset of $H^{-1}(\Omega)$.
Accordingly, we consider excitations in the form of linear functionals. Also, to emphasize that the solution depends on the diffusion coefficient and the excitation, we
write $u_\sigma^l$ in the following. 
The left image in figure \ref{fig:diffusion_example2} illustrates the concentration resulting from a constant source term $g=\chi_{D}$, i.e.\ $l(v)=\int_{D} v\dx$, 
where $D=D_2\cup D_4\cup D_5\cup D_7$ is a union of four circular subdomains as sketched in the right image of figure \ref{fig:diffusion_example}. 
The right image in figure \ref{fig:diffusion_example2} shows the corresponding plot for $D=D_1\cup D_3\cup D_6\cup D_8$. Both images show the solution of \eqref{eq:diffusion} with constant diffusion coefficient $\sigma=1$. 

\begin{figure*}
\begin{tabular}{c c}
  \includegraphics[width=0.45\textwidth]{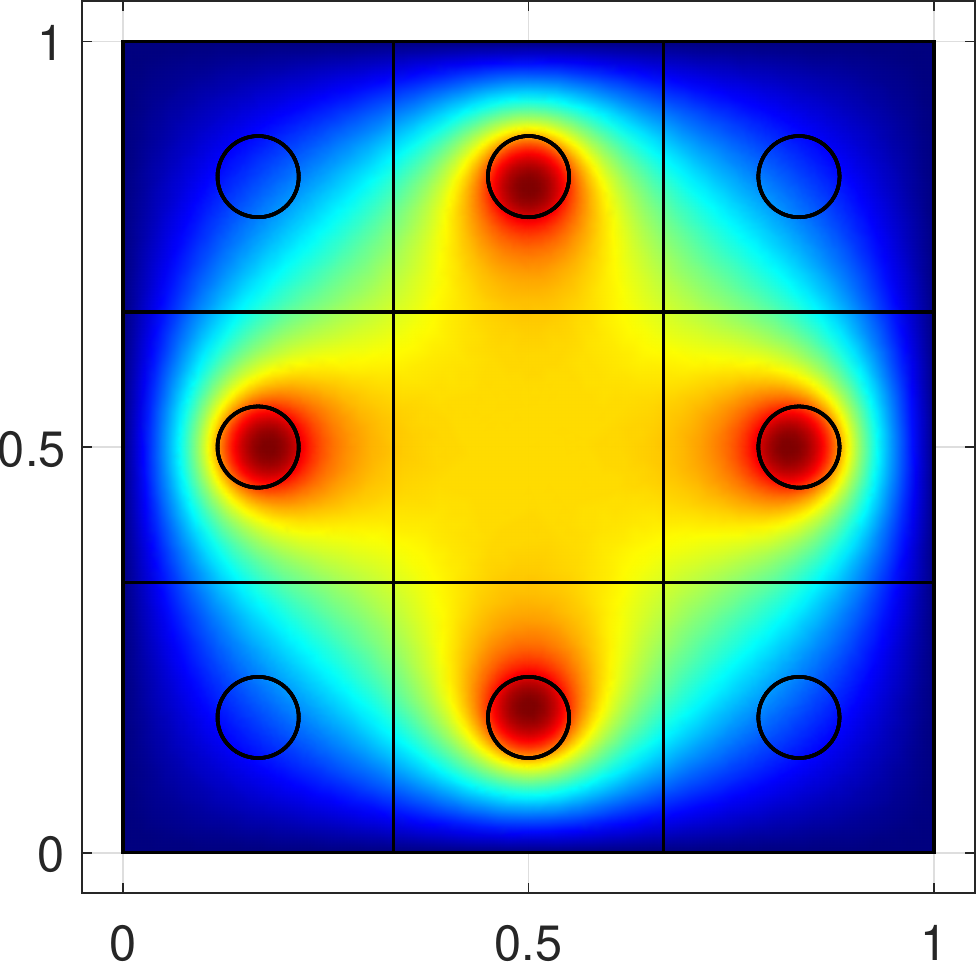}  &
  \includegraphics[width=0.45\textwidth]{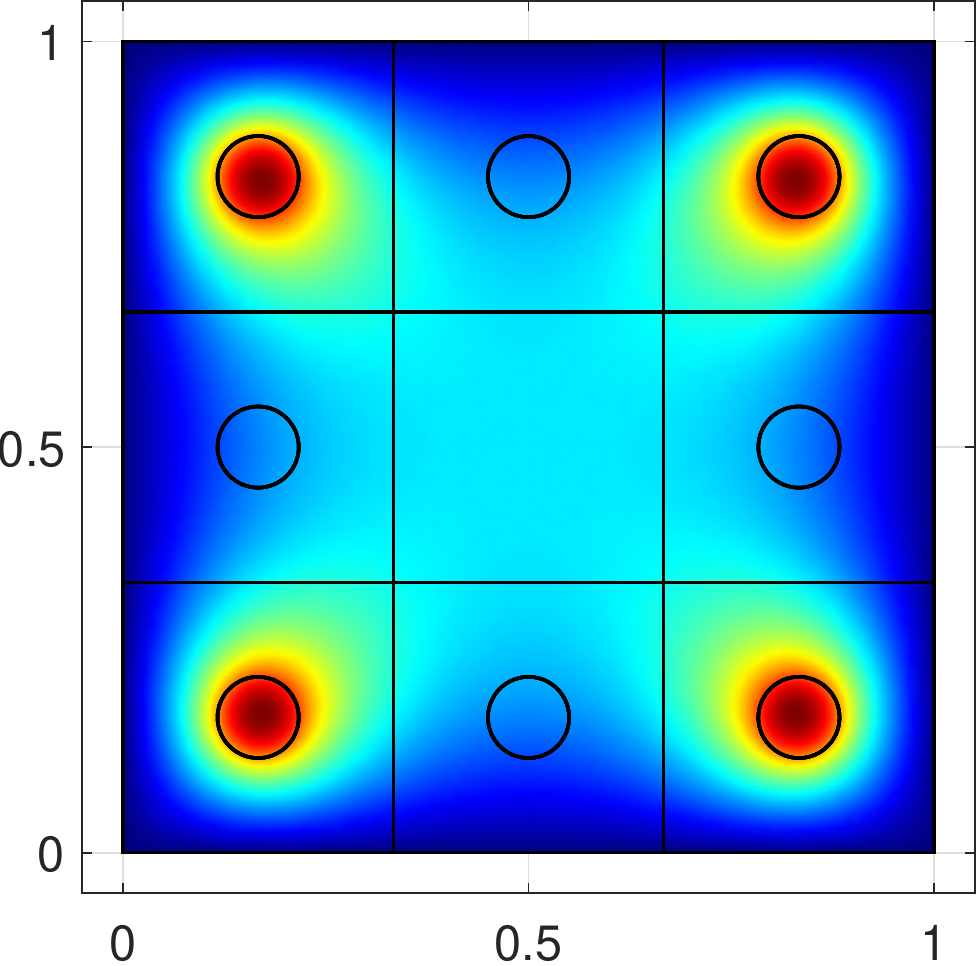}  
\end{tabular}
\caption{PDE solution for source terms on circular subdomains.}
\label{fig:diffusion_example2}       %
\end{figure*}

Natural models for measuring the solution of \eqref{eq:diffusion} also yield to linear functionals. Measuring the total concentration in one of the circular subdomains
$D_j$ corresponds to measuring $r(u):=\int_{D_j} u \dx$. Hence, the inverse problem of determining finitely many information about
the diffusivity coefficient from finitely many measurements of the concentration (possibly but not necessarily resulting from different excitations) leads to
the finite-dimensional inverse problem to
\begin{align*}
\text{determine} \quad \sigma\in \R^n_+ \quad \text{ from } \quad \FF(\sigma)\in \R^m,
\end{align*}
where
\[
\FF:\ \R^n_+\to \R^m, \quad \FF(\sigma):=(r_j(u_\sigma^{l_j}))_{j=1}^m,
\]
and $u_\sigma^{l_j}\in H_0^1(\Omega)$ solves
\[
\sum_{i=1}^n \sigma_i b_i(u_\sigma^{l_j},v)=l_j(v) \quad  \text{ for all } v\in H_0^1(\Omega),
\]
with given $l_j,r_j\in H^{-1}(\Omega)$, $j=1,\ldots,m$, and 
\[
b_i(u,v):=\int_{\mathcal P_i} \nabla u\cdot \nabla v \dx, \quad i=1,\ldots,n.
\]


\subsection{Electrical Impedance Tomography (EIT)} 

We give another example for an application that leads to an inverse elliptic coefficient problems in a similar
form as the diffusion example.

EIT aims to image the inner conductivity structure of a subject by current and voltage measurements 
through electrodes attached to the imaging subject. Let $\Omega\subseteq \R^d$, $d\in \{2,3\}$, be a smoothly bounded domain denoting the imaging subject.
The electrodes $\mathcal E_k$, $k=1,\ldots,K$, are assumed to be open connected subsets of $\partial \Omega$ with disjoint closures. 

When currents with strength $J=(J_1,\ldots,J_K)\in \R^K$ are driven through the $K$ electrodes (with $\sum_{k=1}^K J_k=0$), the resulting electric potential
$u\in H^1(\Omega)$ inside $\Omega$, and the potential $U\in \R^K$ on the electrodes, solve
\begin{alignat*}{2}
\nabla \cdot (\sigma \nabla u)&=0 \quad && \text{ in } \Omega,\\
\sigma \partial_\nu u&=0 \quad && \text{ on } \partial \Omega\setminus \bigcup_{k=1}^K \mathcal E_k,\\
u+z \sigma \partial_\nu &u=\text{const.}=:U_k \quad && \text{ on } \mathcal E_k,\ k=1,\ldots,K,\\
\int_{\mathcal E_k} \sigma \partial_\nu |_{\mathcal E_k} \dx[s]&=J_k && \text{ on } \mathcal E_k,\ k=1,\ldots,K,
\end{alignat*}
where $\sigma\in L^\infty_+(\Omega)$ is the conductivity inside $\Omega$, and $z>0$ is the contact impedance of the electrodes.

Under the gauge condition $U\in \R_\diamond^K:=\{ V\in \R^K:\ \sum_{k=1}^K V_k=0\}$, 
one can show (see \cite{somersalo1992existence}) that this so-called complete electrode model (CEM) for EIT is equivalent to the variational formulation
that $(u,U)\in H^1(\Omega)\times \R^K_\diamond$ solves
\begin{equation}
\int_\Omega \sigma \nabla u\cdot \nabla w \dx + \sum_{k=1}^K \int_{\mathcal E_k} \frac{1}{z}(u-U_k)(w-W_k)\dx[s]
=\sum_{k=1}^K J_k W_k
\end{equation}
for all $(w,W)\in H^1(\Omega)\times \R^K_\diamond$, and unique solvability follows from the Lax-Milgram theorem.

We assume that $z>0$ is known, and that $\sigma(x)=\sum_{i=1}^n \sigma_i \chi_{\mathcal P_i}(x)$ is piecewise constant
with respect to a pixel partition $\Omega=\bigcup_{i=1}^n \mathcal  P_i$, and write $\sigma=(\sigma_1,\ldots,\sigma_n)\in \R^n$
for the unknown conductivity values inside $\Omega$. 

The applied current patterns $J=(J_1,\ldots,J_K)\in \R^K$ can be identified with the functional 
\[
l\in H',\quad l(w,W):=\sum_{k=1}^K J_k W_k \quad \text{ for all } (w,W)\in H:=H^1(\Omega)\times \R^K_\diamond.
\]
Likewise, measuring the voltage between the $k_1$-th and the $k_2$-th electrode corresponds to measuring the linear functional 
\[
r\in H', \quad r(u,U):=U_{k_1}-U_{k_2},
\]
of the solution $(u,U)\in H$ generated by some current pattern.

Hence, the problem of determining the interior conductivity with a fixed finite resolution from finitely many voltage-current measurements in 
EIT (with CEM) leads to the finite-dimensional inverse problem
to 
\begin{align*}
\text{determine} \quad \sigma\in \R^n_+ \quad \text{ from } \quad \FF(\sigma)\in \R^m,
\end{align*}
where $\FF:\ \R^n_+\to \R^m, \quad \FF(\sigma):=(r_j(u_\sigma^{l_j},U_\sigma^{l_j}))_{j=1}^m$, and
$(u_\sigma^{l_j},U_\sigma^{l_j})\in H$ solves
\begin{align*}
b_0((u_\sigma^{l_j},U_\sigma^{l_j}),(w,W))+ \sum_{i=1}^n \sigma_i \, b_i((u_\sigma^{l_j},U_\sigma^{l_j}),(w,W))
 = l_j(w,W)
\end{align*}
for all $(w,W)\in H$, with given $l_j,r_j\in H'$, $j=1,\ldots,m$, and 
\begin{align*}
b_0((u,U),(w,W))&:=\sum_{k=1}^K \int_{\mathcal E_k} \frac{1}{z}(u-U_k)(w-W_k)\dx[s],\\
b_i((u,U),(w,W))&:=\int_{\mathcal P_i} \nabla u\cdot \nabla w \dx.
\end{align*}
Clearly, one could also extend this formulation to cover the case of unknown contact impedances.

\section{The true-solution setting}\label{sect:true_solution}

The examples in section \ref{sec:Examples} lead to inverse problems for a finite-dimensional
non-linear forward operator $\FF:\ \R^n_+\to \R^m$, where evaluations of $\FF$ require solving an infinite-dimensional 
linear problem (the PDE). In this section, we will first derive some properties of $\FF$
for the case that it is defined with the true infinite-dimensional PDE solution. The properties of the operator
$F\approx \FF$, that is defined with a FEM-approximation of the PDE solution, will be studied in section~\ref{sect:FEM_setting}.

\subsection{The true-solution forward operator and its derivative}

We will study problems that appear in the variational formulation of elliptic PDEs with piecewise constant coefficients on a fixed pixel partition, 
as in the examples in section~\ref{sec:Examples}.

\paragraph{The variational setting.}
Let $H$ be a Hilbert space. We consider the problem of finding $u\in H$ that solves
\begin{equation}\label{eq:varform}
b_\sigma(u,v)=l(v),
\end{equation}
where $b_\sigma:\ H\times H\to \R$ is a bilinear form, and $l\in H'=\LL(H,\R)$. $b_\sigma$ is assumed to linearly depend on $n$ parameters $\sigma=(\sigma_1,\ldots,\sigma_n)\in \R^n$ in the following way
\[
b_\sigma(u,v)=b_0(u,v)+\sum_{i=1}^n \sigma_i b_i(u,v),
\]
where $b_0,\ b_i:\ H\times H\to \R$ are bounded, symmetric and positive semidefinite bilinear forms. Writing $\1:=(1,\ldots,1)^T\in \R^n$, we also assume 
that $b_\1$ is bounded and coercive with constants $\beta,C>0$, i.e., 
\[
C\norm{v}^2\geq b_\1(v,v)=b_0(v,v)+\sum_{i=1}^n b_i(v,v)\geq \beta \norm{v}^2 \quad \forall v\in H.
\]
Clearly, this yields that for all $\sigma\in \R^n_+$
\begin{equation}\label{eq:b_sigma_coercive}
C\max\{1,\sigma_1,\ldots,\sigma_n\} \norm{v}^2 \geq  b_\sigma(v,v) \geq \beta \min\{1,\sigma_1,\ldots,\sigma_n\}\norm{v}^2 \quad \forall v\in H,
\end{equation}
so that $b_\sigma$ is symmetric, bounded and coercive. Here and in the following
$\R^n_+$ denotes the set of all $\sigma\in \R^n$ with $\sigma>0$ and "$>$" and "$\geq$" are understood elementwise on $\R^n$.

\paragraph{The true-solution forward operator.}
We now characterize the derivative of the solution of \eqref{eq:varform} with respect to $\sigma$.
\begin{lemma}\label{lemma:inf_dim_sol}
Let $l\in H'$. The solution operator
\[
\mathcal S:\ \R^n_+\to H, \quad \mathcal S(\sigma):=u_\sigma^l, \quad \text{ where $u_\sigma^l\in H$ solves \eqref{eq:varform},}
\]
is infinitely often Fr\'echet differentiable. Its first derivative 
\[
S':\ \R^n_+\to \LL(\R^n,H)
\]
fulfills that, for all $\sigma\in \R^n_+$ and $\tau\in \R^n$, $S'(\sigma)\tau\in H$ is the unique solution of
\[
b_\sigma( S'(\sigma)\tau, w)
= -\sum_{i=1}^n \tau_i b_i(u_\sigma^l , w ) \quad \forall w\in H.
\]
Also, for $r\in H'$, $\sigma\in \R^n_+$, and $\tau\in \R^n$,
\[
r(u_\sigma^l)=b_\sigma(u_\sigma^{l}, u_\sigma^{r}) \quad \text{ and } \quad 
r\left( \mathcal S'(\sigma)\tau \right)=-\sum_{i=1}^n \tau_i b_i(u_\sigma^{l}, u_\sigma^{r}).
\]
\end{lemma}
\begin{proof}
For $\sigma\in \R^n_+$, the Riesz theorem yields that there exists a unique operator $\BB(\sigma)\in \LL(H,H')$ associated to the bilinear form $b_\sigma(\cdot,\cdot)$, i.e.
\[
\langle \BB(\sigma) u,v \rangle_{H'\times H}=b_\sigma(u,v)\quad \text{ for all } u,v\in H.
\]
Clearly, $\BB(\sigma)$ is symmetric and, by the Lax-Milgram theorem, $\BB(\sigma)$ is invertible with symmetric inverse $\BB(\sigma)^{-1}\in \LL(H',H)$. Hence, 
\eqref{eq:varform} is uniquely solvable, and the solution operator $\mathcal S$ is well-defined.

It is easily checked, that $\BB(\sigma)$ is Fr\'echet differentiable for every $\sigma\in \R^n_+$, and that its derivative 
$\BB'(\sigma)\in \LL(\R^n,\LL(H,H'))$ is given by
\[
\BB'(\sigma)\tau=\sum_{i=1}^n \tau_i \BB_i \quad \text{ for all } \sigma\in \R^n_+,\ \tau\in \R^n,
\]
where $\BB_i\in \LL(H,H')$ is the unique operator fulfilling
\[
(\BB_i u,v)_{H'\times H}=b_i(u,v)\quad \text{ for all } u,v\in H.
\]
Since $\BB'(\sigma)$ does not depend on $\sigma$, this also shows that $\BB(\sigma)$ is infinitely often Fr\'echet differentiable with all second and higher derivatives being zero.

Using the derivative of operator inversion and the product and chain rule for the Fr\'echet derivative, we thus obtain that $\mathcal S(\sigma)$ is infinitely often Fr\'echet differentiable with 
\begin{align*}
\mathcal S'(\sigma)\tau&=-\BB(\sigma)^{-1} (\BB'(\sigma)(\tau)) \BB(\sigma)^{-1} l
=-\sum_{i=1}^n \tau_i \BB(\sigma)^{-1} \BB_i u_\sigma^l.
\end{align*}
Hence, $v=\mathcal S'(\sigma)\tau\in H$ solves
\begin{align*}
b_\sigma( v, w)
&=  -\sum_{i=1}^n \langle\tau_i \BB_i u_\sigma^l , w \rangle_{H'\times H}
= -\sum_{i=1}^n \tau_i b_i(u_\sigma^l , w ) \quad \forall w\in H.
\end{align*} 
Moreover, we obtain for all $r\in H'$, by using the symmetry of $\BB(\sigma)$,
\begin{align*}
r( u_\sigma^l )=\langle \BB(\sigma) \BB(\sigma)^{-1} r, u_\sigma^l \rangle_{H'\times H}
= b_\sigma( u_\sigma^l , u_\sigma^r), 
\end{align*}
and
\begin{align*}
r\left( \mathcal S'(\sigma)\tau \right)
&= \langle r, \mathcal S'(\sigma)\tau \rangle_{H'\times H}
= \langle \BB(\sigma) \mathcal S'(\sigma)\tau , \BB(\sigma)^{-1} r\rangle_{H'\times H}\\
&= -\sum_{i=1}^n \tau_i \langle \BB_i\, u_\sigma^l, u_\sigma^r,\rangle_{H'\times H}
= -\sum_{i=1}^n \tau_i b_i(u_\sigma^l, u_\sigma^r),
\end{align*}
which finished the proof.\hfill $\Box$
\end{proof}


\begin{corollary}\label{cor:inf_dim_single_meas}
Let $l,r\in H'$. Then the mapping
\[
\FF_{l,r}:\ \R^n_+\to \R,\quad \FF_{l,r}(\sigma):=r(u_\sigma^l)
\]
fulfills 
\[
\FF_{l,r}(\sigma)=b_\sigma(u_\sigma^l,u_\sigma^r) \quad \text{ for all $\sigma\in \R^n_+$.}
\]
Moreover, $\FF_{l,r}:\ \R^n_+\to \R$ is infinitely often differentiable and its first derivatives fulfill
\[
\frac{\partial}{\partial \sigma_i} \FF_{l,r}(\sigma)=-b_i(u_\sigma^l,u_\sigma^r).
\]
\end{corollary}
\begin{proof}
This follows from Lemma~\ref{lemma:inf_dim_sol}.\hfill $\Box$
\end{proof}

\subsection{Convexity and monotonicity for symmetric measurements}\label{subsect:true_symmetric}

A special mathematical structure appears for measurements $\FF_{l,r}$, when $l$ and $r$ are taken from the same subset of $H'$, and all combinations are used. In the stationary diffusion example this corresponds to using the same subsets of $\Omega$ 
both for excitations and concentration measurements, in EIT this corresponds to using the same electrodes for voltage and current measurements. 

Given a set of $m\in \N$ excitations/measurements $\{l_1,\ldots,l_m\}\subset H'$, we combine the measurements into a matrix-valued map $\FF:\ \R^n_+\to \R^{m\times m}$
\[
\FF(\sigma)=(\FF_{j,k}(\sigma))_{j,k=1,\ldots,m}\in \R^{m\times m}, \quad \FF_{j,k}(\sigma)=\FF_{l_j,l_k}(\sigma)=l_k(u_\sigma^{l_j}).
\]

As before, we write ''$\geq$'' for the elementwise order on $\R^n$. We also write $\sS_m\subseteq \R^{m\times m}$ for the subset of symmetric $m\times m$-matrices,
and ''$\succeq$'' for the Loewner order on $\sS_m$, i.e. $B\succeq A$ denotes that $B-A$ is positive semi-definite.

\begin{lemma}\label{lemma:infdim_convexity}
$\FF:\ \R^n_+\to \R^{m\times m}$ has the following properties:
\begin{enumerate}[(a)]
\item $\FF$ is infinitely often differentiable. 
\item For all $\sigma\in \R^n_+$, $\FF(\sigma)\in \sS_m$ and $\FF(\sigma) \succeq 0$. $\FF(\sigma)$ is positive definite if $l_1,\ldots,l_m\in H'$ are linearly independent.
\item $\FF$ is monotonically non-increasing, i.e.
\begin{align}
\label{eq:FF_mon} \FF'(\sigma)\tau \preceq 0 \quad \text{ for all } \quad \sigma\in \R^n_+,\ 0\leq \tau\in \R^n,
\end{align}
and for all $\sigma^{(1)},\sigma^{(2)}\in \R^n_+$
\begin{align}
\label{eq:FF_mon2}  \sigma^{(1)}\leq \sigma^{(2)} \quad \text{ implies } \quad \FF(\sigma^{(1)})\succeq \FF(\sigma^{(2)}).
\end{align}
\item $\FF$ is convex, i.e., for all $\sigma,\sigma^{(0)}\in \R^n_+$,
\begin{align}
\label{eq:FF_conv} \FF(\sigma)-\FF(\sigma^{(0)}) \succeq \FF'(\sigma^{(0)})(\sigma-\sigma^{(0)}),
\end{align}
and, for all $t\in [0,1]$,
\begin{align}
\label{eq:FF_conv2} \FF((1-t) \sigma^{(0)}+t \sigma)\preceq  (1-t) \FF(\sigma^{(0)})+t \FF(\sigma).
\end{align}
\end{enumerate}
\end{lemma}
\begin{proof}
Corollary \ref{cor:inf_dim_single_meas} shows that each component of $\FF$ is infinitely often differentiable so that (a) is proven.

For the rest of the proof let $\sigma\in \R^n_+$, $g\in \R^m$, and set $l:=\sum_{j=1}^m g_j l_j$.
By corollary~\ref{cor:inf_dim_single_meas}, 
\[
l_k(u_\sigma^{l_j})=b_\sigma(u_\sigma^{l_j},u_\sigma^{l_k})=b_\sigma(u_\sigma^{l_k},u_\sigma^{l_j})=l_j(u_\sigma^{l_k}),
\]
so that $\FF(\sigma)$ is a symmetric matrix. Moreover, 
\[
g^T \FF(\sigma) g=\sum_{j,k=1}^m g_j l_k(u_\sigma^{l_j}) g_k = \sum_{j,k=1}^m g_j g_k b_\sigma(u_\sigma^{l_j},u_\sigma^{l_k})=b_\sigma(u_\sigma^l,u_\sigma^l)\geq 0,
\]
so that $\FF(c)\succeq 0$. If $g\neq 0$ and $l_1,\ldots,l_m\in H'$ are linearly independent then $l\neq 0$, which implies 
$u_\sigma^l\neq 0$ and thus $g^T \FF(\sigma) g>0$. Hence, (b) is proven. 

To prove (c) and (d), we start by using again corollary~\ref{cor:inf_dim_single_meas} and obtain
\[
g^T (\FF'(\sigma)\tau) g=-\sum_{i=1}^n \tau_i b_i(u_\sigma^{l},u_\sigma^{l}) \quad \text{ for all } \tau\in \R^n.
\]
Since the bilinear forms $b_i(\cdot,\cdot)$ are positive semi-definite, this implies \eqref{eq:FF_mon}.

To prove \eqref{eq:FF_conv}, let $\sigma^{(0)}\in \R^n_+$. For brevity we write $u_0^l:=u_{\sigma_0}^l$.
Using 
\[
b_\sigma(u_\sigma^l,u_0^l)=l(u_0^l)=b_{\sigma_0}(u_0^l,u_0^l),
\]
we obtain that
\begin{align*}
0&\leq b_\sigma(u_\sigma^l-u_0^l,u_\sigma^l-u_0^l)=b_\sigma(u_\sigma^l,u_\sigma^l)-2b_\sigma(u_c^l,u_0^l)+b_\sigma(u_0^l,u_0^l)\\
&=b_\sigma(u_\sigma^l,u_\sigma^l)-2b_{\sigma_0}(u_0^l,u_0^l)+b_\sigma(u_0^l,u_0^l)\\
&= g^T (\FF(\sigma)-\FF(\sigma^{(0)})) g +b_\sigma(u_0^l,u_0^l)-b_{\sigma_0}(u_0^l,u_0^l)\\
&= g^T (\FF(\sigma)-\FF(\sigma^{(0)}) g + \sum_{i=1}^n (\sigma_i-\sigma_i^{(0)}) b_i(u_0^l,u_0^l).
\end{align*}
This shows that
\[
g^T (\FF(\sigma)-\FF(\sigma^{(0)})) g\geq -\sum_{i=1}^n (\sigma_i-\sigma_i^{(0)}) b_i(u_0^l,u_0^l)=g^T \FF'(\sigma^{(0)})(\sigma-\sigma^{(0)})g,
\]
so that \eqref{eq:FF_conv} holds. Together with \eqref{eq:FF_mon} this also implies \eqref{eq:FF_mon2}.

\eqref{eq:FF_conv2} follows from \eqref{eq:FF_conv} by the following standard argument. Let 
$\sigma,\sigma^{(0)}\in \R^n_+$, $t\in [0,1]$, and set 
\[
\sigma^{(t)}:=t \sigma+ (1-t) \sigma^{(0)}\in \R^n_+.
\]
Using \eqref{eq:FF_conv} on $\FF(\sigma)-\FF(\sigma^{(t)})$ and $\FF(\sigma^{(0)})-\FF(\sigma^{(t)})$, we 
then obtain that
\begin{align*}
\lefteqn{(1-t) \FF(\sigma^{(0)})+t \FF(\sigma)-\FF(\sigma^{(t)})}\\
&=(1-t) (\FF(\sigma^{(0)})-\FF(\sigma^{(t)})) +t ( \FF(\sigma)-\FF(\sigma^{(t)}))\\
&\succeq (1-t)\FF'(\sigma^{(t)})(\sigma^{(0)}-\sigma^{(t)}) + t\FF'(\sigma^{(t)})(\sigma-\sigma^{(t)})\\
&= \FF'(\sigma^{(t)})((1-t)\sigma^{(0)}+ t \sigma -\sigma^{(t)} )=0,
\end{align*}
which proves \eqref{eq:FF_conv2}.\hfill $\Box$
\end{proof}




\section{The FEM setting}\label{sect:FEM_setting}

\subsection{The FEM-approximated forward operator and its derivative}

\paragraph{The Finite Element Method.}
The Finite Element Method numerically approximates the solution of \eqref{eq:varform} by solving it in a finite-dimensional
subspace $V\subset H$, e.g.\ the subspace of continuous, piecewise linear functions on a fixed triangulation. 
Let $\Lambda_1,\ldots,\Lambda_N$ denote a basis of $V$, e.g.\ the so-called hat functions for linear finite elements. 
Then the finite-dimensional variational problem
\begin{equation}\label{eq:VarForm_FEM}
\tilde u_\sigma^l\in V \quad \text{ solves } \quad b_\sigma(\tilde u_\sigma^l,v)=l(v) \quad \text{ for all } v\in V
\end{equation}
is equivalent to 
\begin{equation}\label{eq:LGS_FEM}
\tilde u_\sigma^l=\sum_{j=1}^N \lambda_j \Lambda_j, \quad \text{ where } \quad B_\sigma\lambda=y^l,
\end{equation}
with $\lambda=(\lambda_j)_{j=1}^N\in \R^N$, and the so-called stiffness matrix and load vector 
\begin{align*}
B_\sigma &\in \R^{N\times N}, \hspace{-3em} && \text{ with $(j,k)$-th entry given by } b_\sigma(\Lambda_j,\Lambda_k),\\
y^l&\in \R^N, \hspace{-3em} && \text{ with $j$-th entry given by } l(\Lambda_j).
\end{align*}
It follows from the Lax-Milgram theorem that \eqref{eq:VarForm_FEM} is  uniquely solvable and that $B$ is a symmetric, 
positive definite (and thus invertible) matrix. Moreover, the C\'ea-Lemma yields that the FEM approximation $\tilde u_\sigma^l\in V$
is as good an approximation to the true solution $u_\sigma^l\in H$ as elements of the finite-dimensional space $V$ can be:
\begin{equation}\label{eq:Cea}
\norm{u_\sigma^l-\tilde u_\sigma^l}\leq \frac{C_\sigma}{\beta_\sigma}\inf_{v\in V}\norm{u_\sigma^l-v},
\end{equation}
where $C_\sigma:=C\max\{1,\sigma_1,\ldots,\sigma_n\}$, and $\beta_\sigma:=\beta \min\{1,\sigma_1,\ldots,\sigma_n\}$ are the
continuity and coercivity constants of $b_\sigma$, cf.\ \eqref{eq:b_sigma_coercive}.

\paragraph{Pixel stiffness matrices.}
Finite element software packages include triangulation algorithms, assembling routines for the global stiffness matrix $B_\sigma$ and the load vector $y^l$, and 
efficient solvers for the linear system $B_\sigma\lambda=y^l$. 
For our setting where
\[
b_\sigma(u,v)=b_0(u,v)+\sum_{i=1}^n \sigma_i b_i(u,v),
\]
we will also require the \emph{pixel stiffness matrices} 
\[
B_i\in \R^{N\times N}, \quad \text{ with $(j,k)$-th entry given by } \quad b_i(\Lambda_j,\Lambda_k).
\]

The assembling of $B_\sigma$ is usually done by writing it as a weighted sum of element stiffness matrices.
In our setting, it is natural to assume that the pixel partition complies with the FEM triangulation, i.e., that each pixel is a union of triangulation elements.
Figure \ref{fig:compliant_mesh} shows a coarser and a finer FEM mesh for the diffusion example, both complying with the pixel partition and with the subdomains that are used for measurements and excitations.
Hence, during the assembly of the global stiffness matrix $B_\sigma$, the pixel stiffness matrices can usually be obtained without any additional computational cost
by the simple intermediate step of first summing up the element matrices for each pixel, and then summing up the pixel stiffness matrices to obtain $B_\sigma$.
Alternatively, the pixel stiffness matrix $B_i$ can be conveniently obtained from global stiffness matrices by the simple identities
\[
B_i=B_{\1+e_i}-B_{\1},\quad \text{ and } \quad B_0=B_{\1}-\sum_{i=1}^n B_i,
\]
where $B_{\1+e_i}$ and $B_{\1}$ denote the global stiffness matrix $B_\sigma$ for $\sigma=\1+e_i$ and $\sigma=\1$, respectively, and
$e_i\in \R^n$ is the $i$-th unit vector. Note that this does not require any knowledge of the triangulation details.

\begin{figure*}
\begin{tabular}{c c}
  \includegraphics[width=0.45\textwidth]{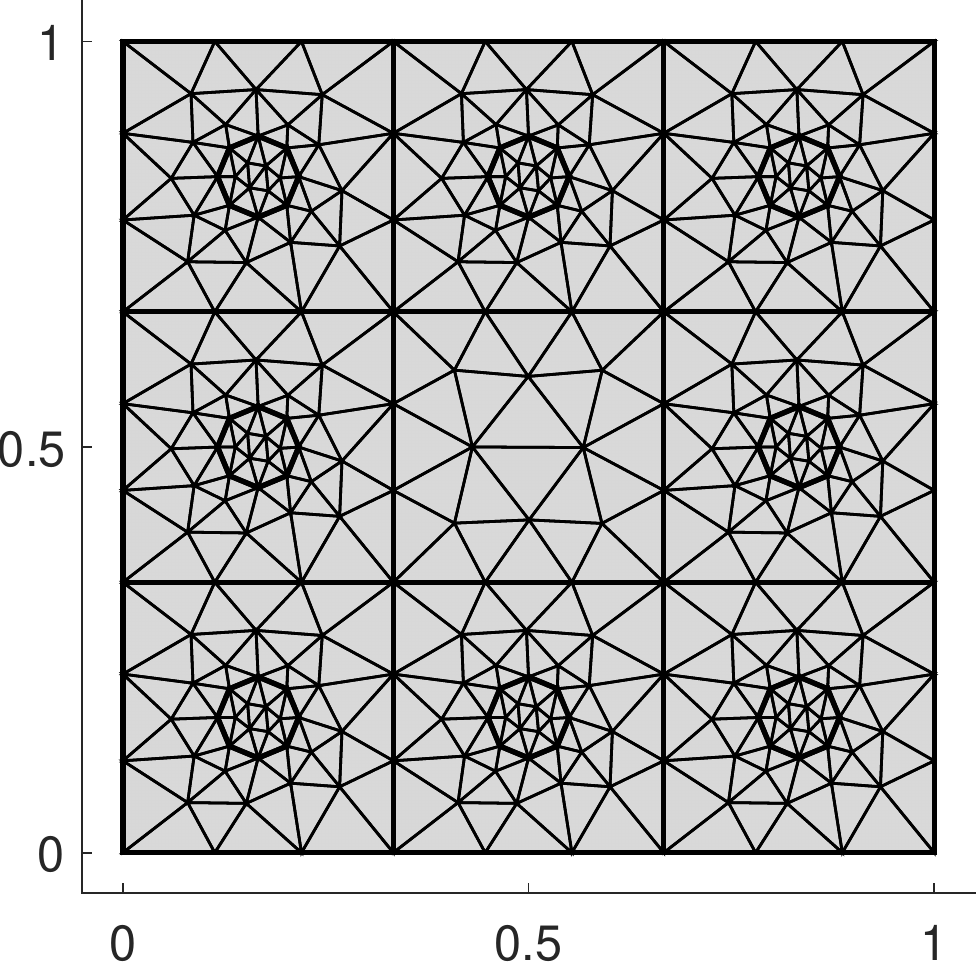} &
  \includegraphics[width=0.45\textwidth]{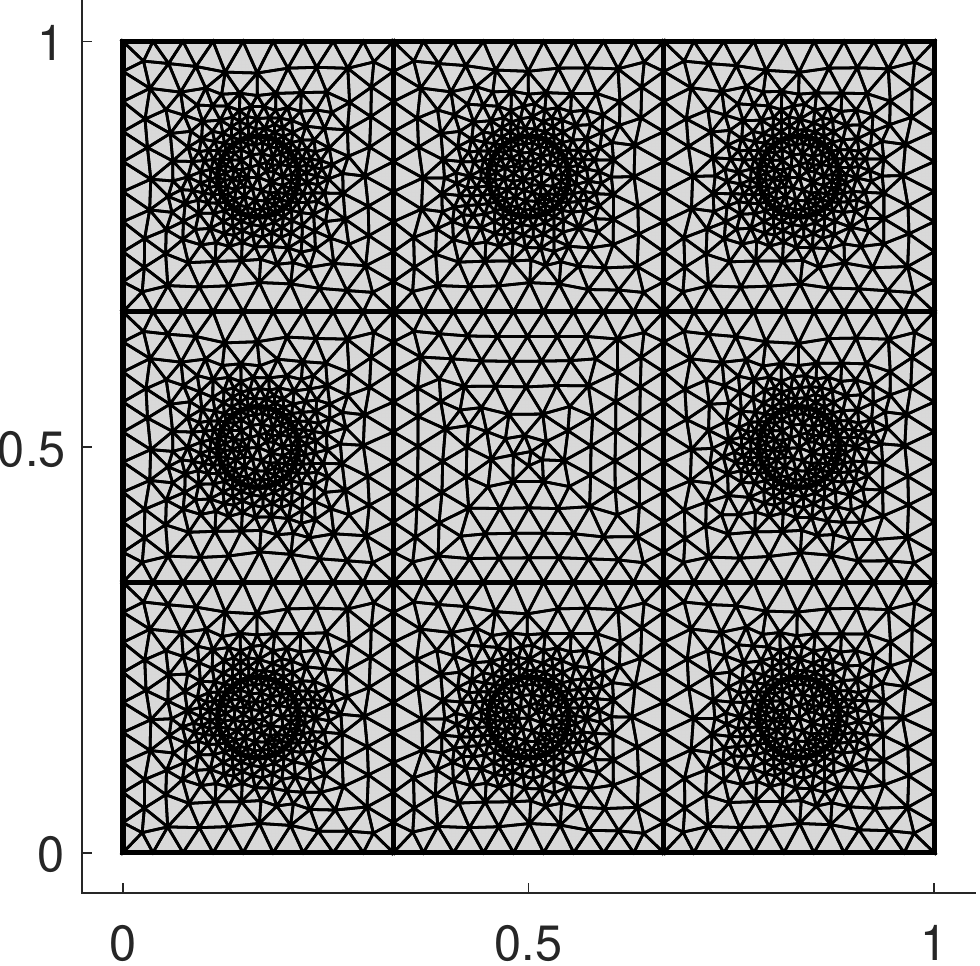} 
\end{tabular}
\caption{A coarser and a finer FEM-mesh for the diffusion example, both complying with the pixel partition and the measurement/excitation subdomains.}
\label{fig:compliant_mesh}       
\end{figure*}


\paragraph{The FEM-approximated forward operator.}
Given $l,r\in H'$, we approximate the true measurement $\FF_{l,r}(\sigma)=r(u_\sigma^l)$ by
\[
F_{l,r}(\sigma):=r(\tilde u_\sigma^l),
\]
where $\tilde u_\sigma^l\in V$ is the FEM-approximation to the true solution $u_\sigma^l\in H$, i.e., the solution of \eqref{eq:VarForm_FEM}.

\begin{lemma}\label{lemma:FEM_single_meas}
Let $l,r\in H'$. Then 
\[
F_{l,r}(\sigma)=b_\sigma(\tilde u_\sigma^l,\tilde u_\sigma^r) \quad \text{ for all $\sigma\in \R^n_+$.}
\]
Moreover, $F_{l,r}:\ \R^n_+\to \R$ is infinitely often differentiable and its first derivatives fulfill 
\[
\frac{\partial}{\partial \sigma_i} F_{l,r}(\sigma)=-b_i(\tilde u_\sigma^l,\tilde u_\sigma^r), \quad i=1,\ldots,n.
\]
\end{lemma}
\begin{proof}
This follows by applying corollary \ref{cor:inf_dim_single_meas} to the Hilbert space $V$.\hfill $\Box$
\end{proof}

From lemma~\ref{lemma:FEM_single_meas}, we obtain a simple FEM-based implementation of the forward operator and its derivative.
\begin{corollary}\label{cor:FEM_single_meas}
With 
\begin{alignat*}{2}
\tilde u_\sigma^l&=\sum_{j=1}^N \lambda^{l}_j \Lambda_j, \qquad & \lambda^l&=(\lambda_j^l)_{j=1}^N\in \R^N,\\
\tilde u_\sigma^r&=\sum_{j=1}^N \lambda^{r}_j \Lambda_j, \qquad & \lambda^r&=(\lambda_j^r)_{j=1}^N\in \R^N,
\end{alignat*}
we have that
\[
F_{l,r}(\sigma)=(\lambda^l)^T B_\sigma \lambda^r=(\lambda^l)^T y^r, \quad \text{ and } \quad \frac{\partial}{\partial \sigma_i} F_{l,r}(\sigma)=- (\lambda^l)^T B_i \lambda^r.
\]
\end{corollary}
\begin{proof}
This follows from lemma~\ref{lemma:FEM_single_meas}.\hfill $\Box$
\end{proof}

We summarize the consequences of corollary~\ref{cor:FEM_single_meas} in algorithm \ref{algo:FEM}. 
Using a FEM package that is capable of solving the considered PDE, and that allows access to the stiffness matrix and the load vector, one can simply implement 
the FEM-approximated forward operator and all its first derivatives by a few lines of extra code. This calculation merely requires solving two linear systems with the stiffness matrix
(which is equivalent to two PDE solutions). 

\begin{algorithm}[H]
\caption{FEM-approximation of $F_{l,r}(\sigma)$ and $\frac{\partial}{\partial \sigma_i} F_{l,r}(\sigma)$, $i=1,\ldots,n$}
\label{algo:FEM}
\begin{algorithmic}
\STATE{\textbf{given} $l,r\in H'$, $\sigma\in \R^n_+$}
\STATE{$\cdot$ use FEM package to calculate load vectors $y^l$ and $y^r$}
\STATE{$\cdot$ use FEM package to calculate stiffness matrices $B_{\1}$, and $B_{\1+e_i}$ for all $i=1,\ldots,n$}
\STATE{$\cdot$ set $B_i:=B_{\1+e_i}-B_{\1}$ for $i=1,\ldots,n$, and $B_\sigma:=B_\1+\sum_{i=1}^n (\sigma_i-1) B_i$}
\STATE{$\cdot$ solve $B_\sigma \lambda^l=y^l$ and $B_\sigma \lambda^r=y^r$ for $\lambda^l$ and $\lambda^r$}
\RETURN{$F_{l,r}(\sigma):=(\lambda^l)^T y^r$ and $\frac{\partial}{\partial \sigma_i}  F_{l,r}(\sigma):=-  (\lambda^l)^T B_i \lambda^r$, $i=1,\ldots,n$}
\end{algorithmic}
\end{algorithm}

\paragraph{Convergence of the FEM-approximated forward operator.}

The following lemma shows that the FEM-approximated operator and its first derivatives agree with their true-solution 
counterparts as good as the FEM solution agrees with the true solution. Hence, by the C\'ea-Lemma \eqref{eq:Cea}, 
$F_{l,r}(\sigma)$ and $\frac{\partial}{\partial \sigma_i} F_{l,r}(\sigma)$ will be as good an approximation to 
$\FF_{l,r}(\sigma)$ and $\frac{\partial}{\partial \sigma_i} \FF_{l,r}(\sigma)$
as the true solutions can be approximated by elements of the finite-dimensional space $V$.

\begin{lemma}\label{FEM_forward_error}
For all $l,r\in H'$ and $\sigma\in \R^n_+$, we have that:
\begin{align}\label{eq:FEM_forward_error1}
\FF_{l,r}(\sigma)-F_{l,r}(\sigma)&=b_\sigma(u_\sigma^l-\tilde u_\sigma^l,u_\sigma^r-\tilde u_\sigma^r),\\
\label{eq:FEM_forward_error2}
\frac{\partial}{\partial \sigma_i} \FF_{l,r}(\sigma)-\frac{\partial}{\partial \sigma_i} F_{l,r}(\sigma)&=b_i(\tilde u_\sigma^l,\tilde u_\sigma^r-u_\sigma^r)+b_i(\tilde u_\sigma^l-u_\sigma^l,u_\sigma^r).
\end{align}
Hence, by the C\'ea-Lemma \eqref{eq:Cea},
\begin{align*}
0\leq \FF_{l,r}(\sigma)-F_{l,r}(\sigma)& \leq C_\sigma \norm{u_\sigma^l-\tilde u_\sigma^l}\norm{u_\sigma^r-\tilde u_\sigma^r}\\
& \leq \frac{C_\sigma^3}{\beta_\sigma^2}\inf_{v\in V}\norm{u_\sigma^l-v} \inf_{v\in V}\norm{u_\sigma^r-v},
\end{align*}
and
\begin{align*}
\lefteqn{\left|\frac{\partial}{\partial \sigma_i} \FF_{l,r}(\sigma)-\frac{\partial}{\partial \sigma_i} F_{l,r}(\sigma)\right|}\\
& \leq C_i \norm{\tilde u_\sigma^l} \norm{\tilde u_\sigma^r-u_\sigma^r} + C_i\norm{u_\sigma^r} \norm{\tilde u_\sigma^l-u_\sigma^l}\\
& \leq  \frac{C_i C_\sigma}{\beta_\sigma}\left( 
\norm{\tilde u_\sigma^l} \inf_{v\in V}\norm{u_\sigma^r-v}+ \norm{ u_\sigma^r} \inf_{v\in V}\norm{u_\sigma^l-v}\right),
\end{align*}
where $C_i>0$ is the continuity constant of $b_i(\cdot,\cdot)$.
\end{lemma}
\begin{proof}
Using 
\begin{align*}
b_\sigma(\tilde u_\sigma^l,\tilde u_\sigma^r)=l(\tilde u_\sigma^r)=b_\sigma(u_\sigma^l,\tilde u_\sigma^r),\quad \text{ and } \quad
b_\sigma(\tilde u_\sigma^l,\tilde u_\sigma^r)=r(\tilde u_\sigma^l)=b_\sigma(\tilde u_\sigma^l,u_\sigma^r),
\end{align*}
we obtain \eqref{eq:FEM_forward_error1} from
\begin{align*}
\FF_{l,r}(\sigma)-F_{l,r}(\sigma)&=b_\sigma(u_\sigma^l,u_\sigma^r)-b_\sigma(\tilde u_\sigma^l,\tilde u_\sigma^r)=b_\sigma(u_\sigma^l,u_\sigma^r-\tilde u_\sigma^r)\\
&=b_\sigma(u_\sigma^l,u_\sigma^r-\tilde u_\sigma^r)-b_\sigma(\tilde u_\sigma^l,u_\sigma^r-\tilde u_\sigma^r)\\
&=b_\sigma(u_\sigma^l-\tilde u_\sigma^l,u_\sigma^r-\tilde u_\sigma^r).
\end{align*}
Also, 
\begin{align*}
\frac{\partial}{\partial \sigma_i} \FF_{l,r}(\sigma)-\frac{\partial}{\partial \sigma_i} F_{l,r}(\sigma)
&= b_i(\tilde u_\sigma^l,\tilde u_\sigma^r)-b_i(u_\sigma^l,u_\sigma^r)\\
&= b_i(\tilde u_\sigma^l,\tilde u_\sigma^r-u_\sigma^r)+b_i(\tilde u_\sigma^l-u_\sigma^l,u_\sigma^r),
\end{align*}
which shows \eqref{eq:FEM_forward_error2}. \hfill $\Box$
\end{proof}

\subsection{Convexity and monotonicity for symmetric measurements}\label{subsect:FEM_symmetric}

As in subsection \ref{subsect:true_symmetric} we now consider the symmetric measurement case, where $l$ and $r$ are taken from the same subset of $H'$ (and all combinations are used). 
Given a set of $m\in \N$ excitations/measurements $\{l_1,\ldots,l_m\}\subset H'$, we combine the measurements into a matrix-valued map $F:\ \R^n_+\to \R^{m\times m}$
\[
F(\sigma)=(F_{j,k}(\sigma))_{j,k=1,\ldots,m}\in \R^{m\times m}, \quad F_{j,k}(\sigma)=F_{l_j,l_k}(\sigma)=l_k(u_\sigma^{l_j}).
\]

The entries of $F(\sigma)$ and its first derivatives $\frac{\partial}{\partial \sigma_i} F(\sigma)$, $i=1,\ldots,n$
can be calculated as in algorithm~\ref{algo:FEM}. Let us stress that this approach is particularly efficient in this symmetric case as it requires only $m$ linear system solutions with the stiffness matrix (i.e., the equivalent of $m$ PDE solutions) for calculating all $m^2$ entries of $F(\sigma)\in \R^{m\times m}$ and all $n m^2$ entries of the $n$ matrices $\frac{\partial}{\partial \sigma_i} F(\sigma)\in \R^{m\times m}$.

As in subsection \ref{subsect:true_symmetric}, the FEM-approximated forward operator 
is monotonically non-increasing and convex in the sense of the elementwise order ''$\geq$'' on $\R^n$, and 
the Loewner order  ''$\succeq$'' on the set of symmetric $m\times m$-matrices. 

\begin{lemma}\label{lemma:FEM_symmetric}
$F:\ \R^n_+\to \R^{m\times m}$ has the following properties:
\begin{enumerate}[(a)]
\item $F$ is infinitely often differentiable. 
\item For all $\sigma\in \R^n_+$, $F(\sigma)\in \sS_m$ and $F(\sigma) \succeq 0$. $F(\sigma)$ is positive definite if $l_1,\ldots,l_m\in H'$ are linearly independent.
\item $F$ is monotonically non-increasing, i.e.
\begin{align}
\label{eq:F_mon} F'(\sigma)\tau \preceq 0 \quad \text{ for all } \quad \sigma\in \R^n_+,\ 0\leq \tau\in \R^n,
\end{align}
\item $F$ is convex, i.e.
\begin{align}
\label{eq:F_conv} F(\sigma)-F(\sigma^{(0)}) \succeq F'(\sigma^{(0)})(\sigma-\sigma^{(0)}) \quad \text{ for all } \quad \sigma,\sigma^{(0)}\in \R^n_+.
\end{align}
\item $\FF(\sigma)\succeq F(\sigma)$.
\end{enumerate}
\end{lemma}
\begin{proof}
(a)--(d) follow from applying lemma \ref{lemma:infdim_convexity} on the Hilbert space $V$. (e) was proven in lemma~\ref{FEM_forward_error}. \hfill $\Box$
\end{proof}

\section{Numerical examples and inverse problem challenges}\label{sect:numerics}

In this section, we will show some numerical results for the stationary diffusion example from section~\ref{subsect:diffusion} 
and demonstrate some major challenges that arise in solving the inverse coefficient problem of recovering $\hat \sigma\in \R^n$ from
$\FF(\hat \sigma)\in \R^m$, or from a noisy version $Y^\delta\approx F(\hat \sigma)$. The source codes for the following examples (and also for generating figure \ref{fig:diffusion_example} and \ref{fig:diffusion_example2}) 
are given in the appendix for the reader's reference. 

\subsection{Non-uniqueness}
Even for $m\geq n$, and a noise-free measurement $\hat Y=F(\hat \sigma)\in \R^m$, it is not clear whether the measurements uniquely determine the unknown $\hat \sigma\in \R^n$.
To demonstrate this on a simple one-dimensional example, let us consider the stationary diffusion example
with $3\times 3$ pixels and circular excitation/measurement subdomains in each boundary pixel as in figure \ref{fig:diffusion_example}. We apply a source term in $D_1$ in the lower left pixel, and measure the resulting total concentration in $D_8$ in the top right pixel, so that $l=\chi_1\in H^{-1}(\Omega)$ and $r=\chi_8\in H^{-1}(\Omega)$, where we write $\chi_j:=\chi_{D_j}$ for the ease of notation. We choose $ \sigma=1$ in all pixels except $\mathcal P_i$, and on $\mathcal P_i$ we vary the diffusivity in steps of $0.01$ up to $3$.
Figure \ref{fig:non_uniqueness} shows $F_{l,r}(\sigma)$ for all $i=1,\ldots,9$, in the same order as the pixels, 
e.g., the lower left image shows $F_{l,r}(\sigma)$ for $\sigma=(\sigma_1,1,\ldots,1)$ for varying $\sigma_1$.

\begin{figure*}
\begin{tabular}{c c c}
  \includegraphics[width=0.28\textwidth]{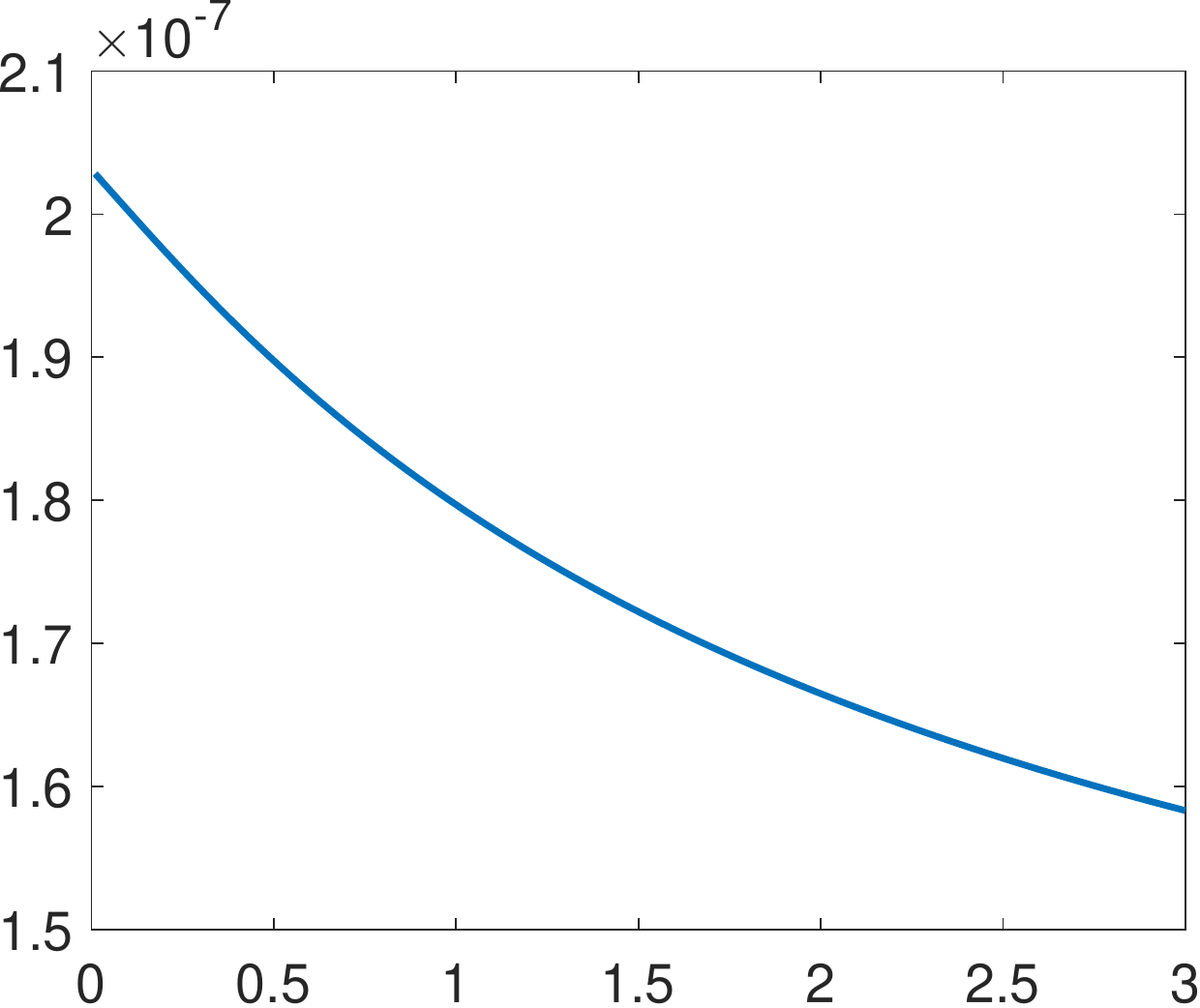} &
  \includegraphics[width=0.28\textwidth]{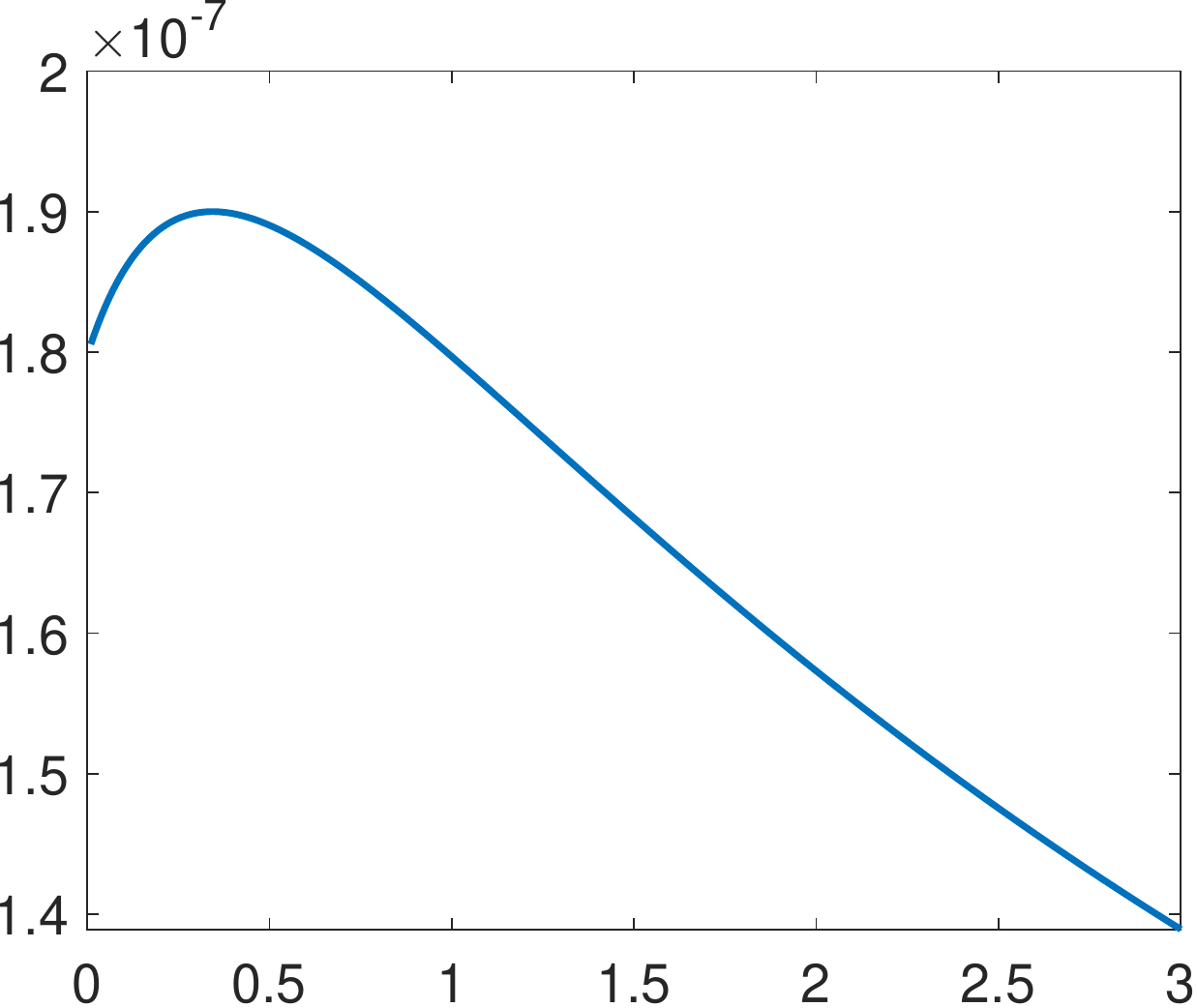} &
  \includegraphics[width=0.28\textwidth]{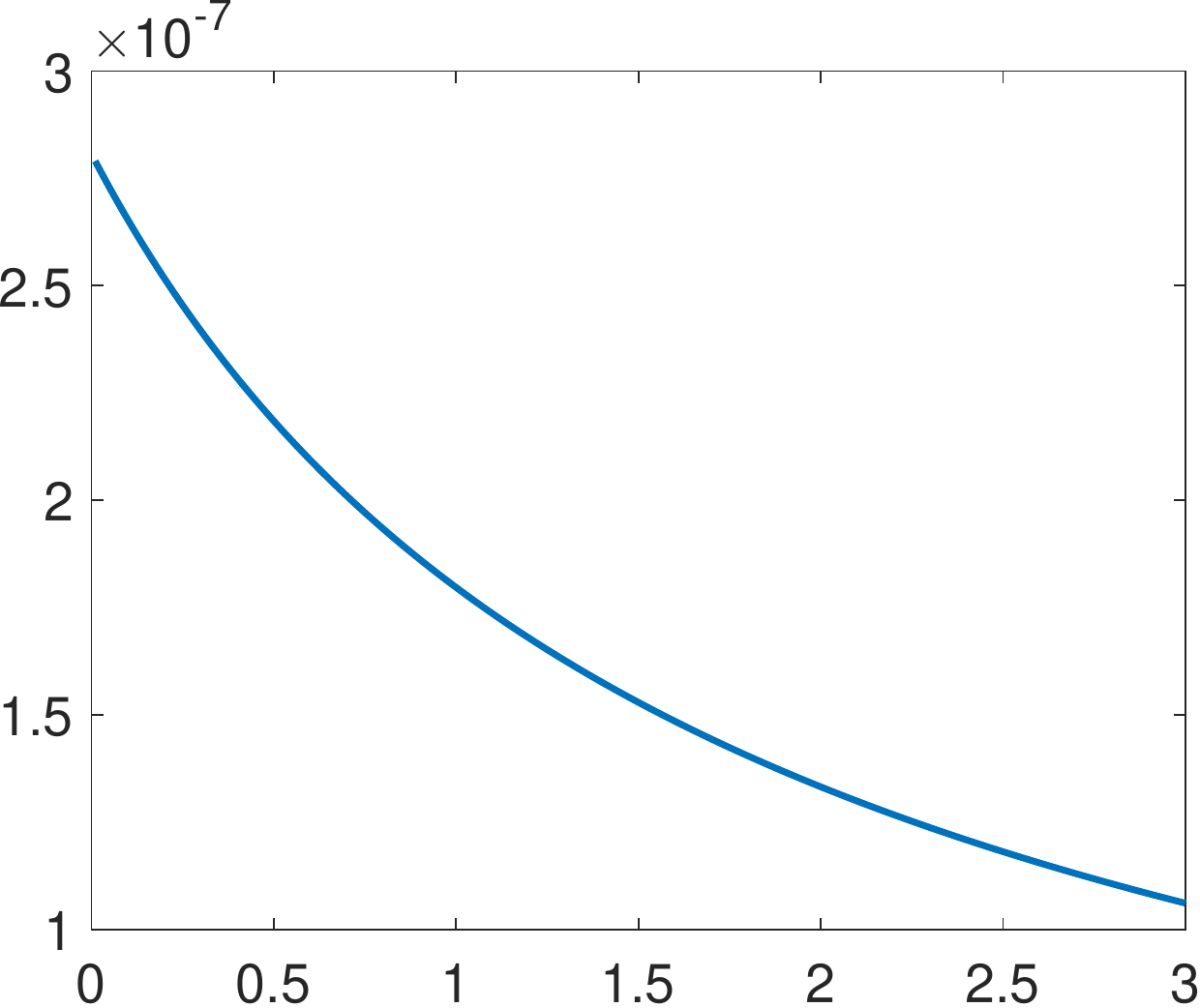}\\
  \includegraphics[width=0.28\textwidth]{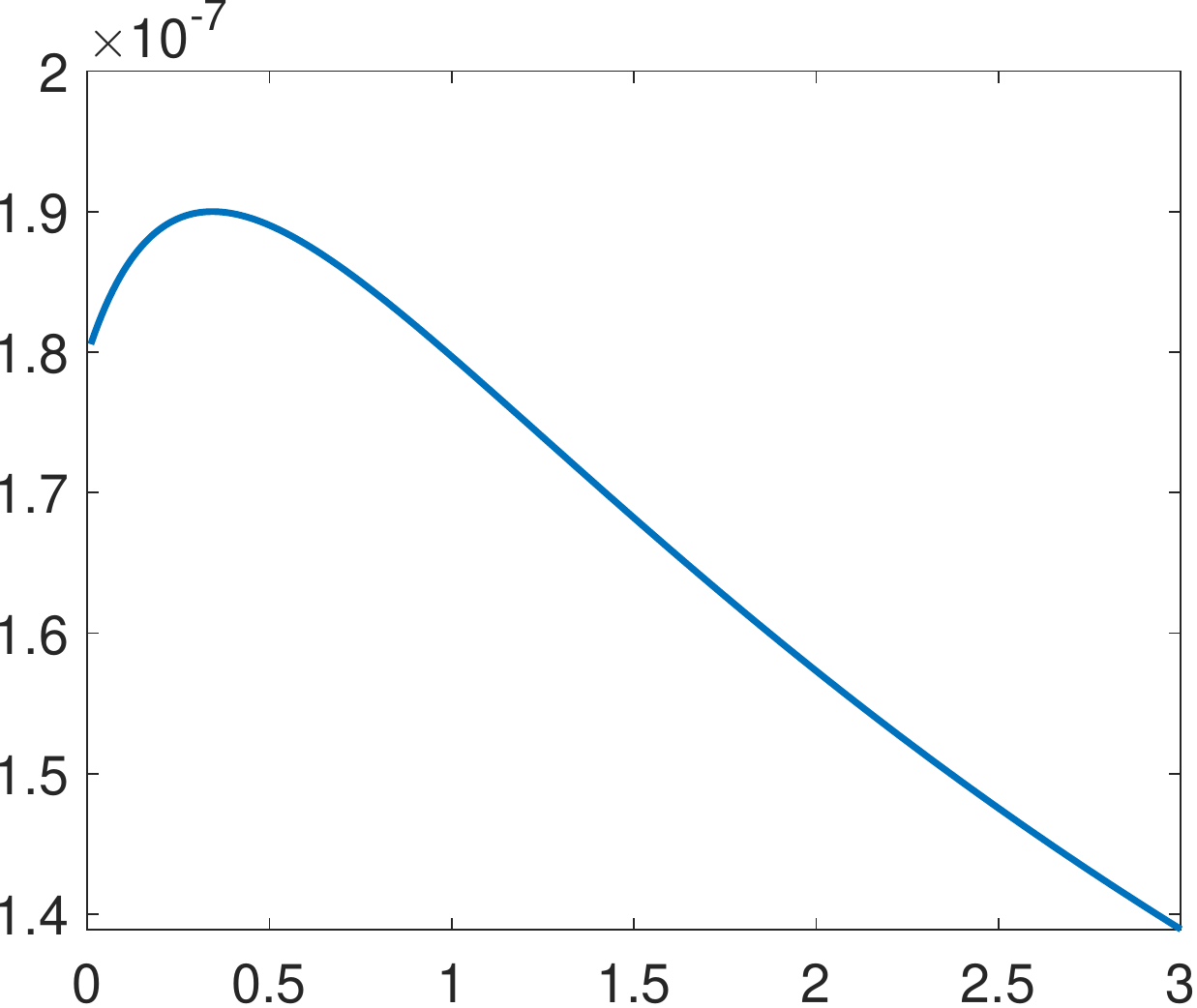} &
  \includegraphics[width=0.28\textwidth]{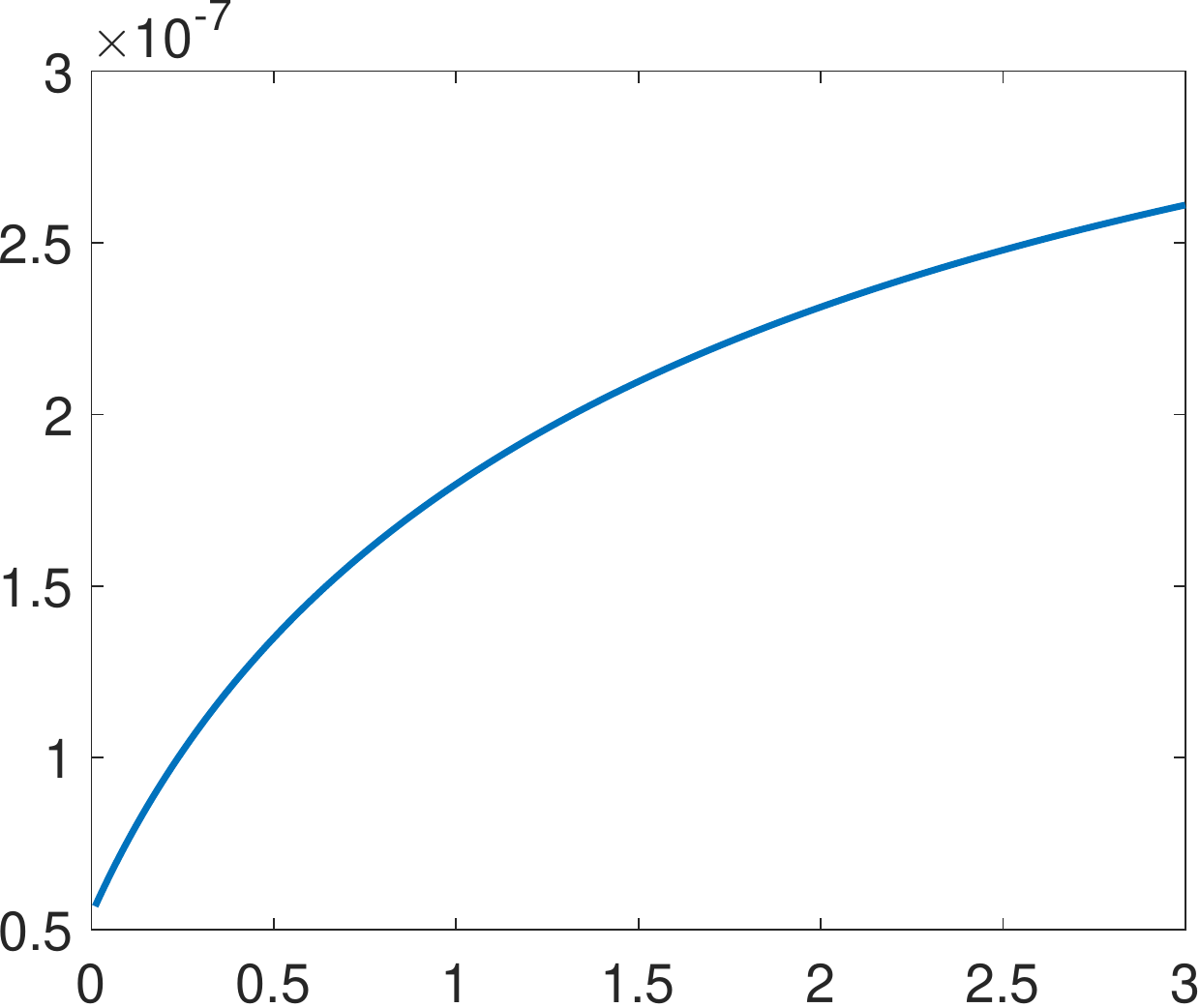} &
  \includegraphics[width=0.28\textwidth]{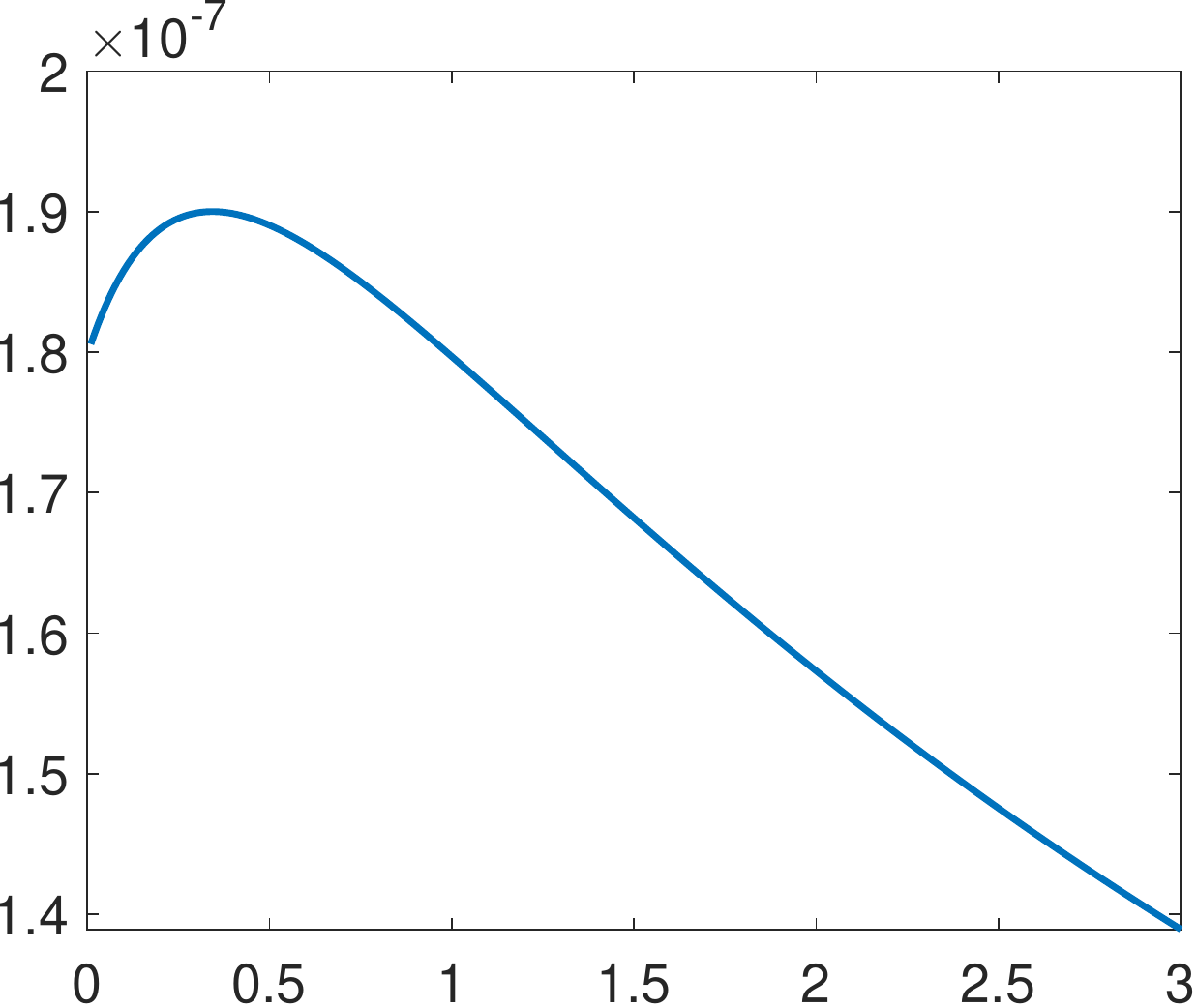}\\
  \includegraphics[width=0.28\textwidth]{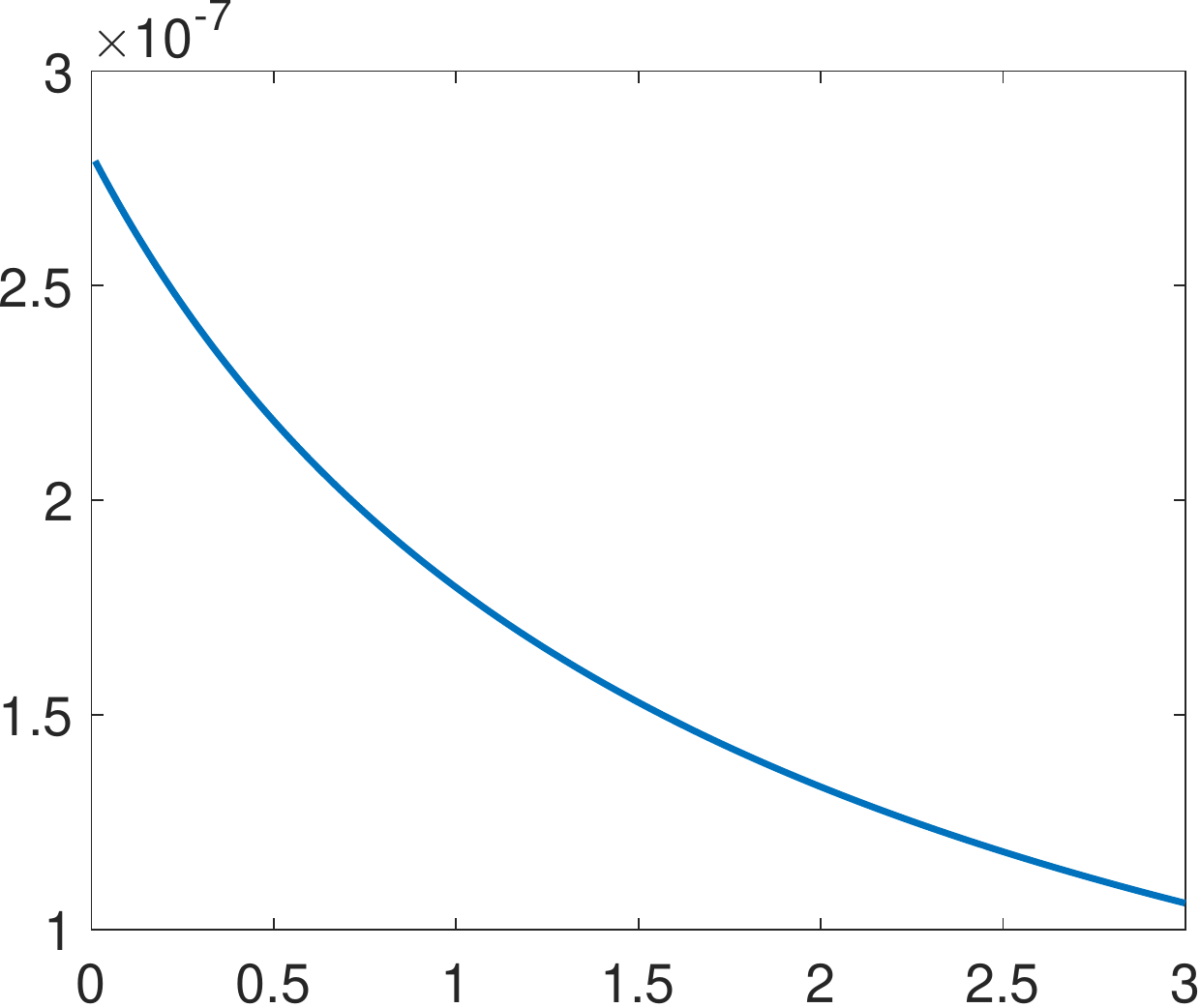} &
  \includegraphics[width=0.28\textwidth]{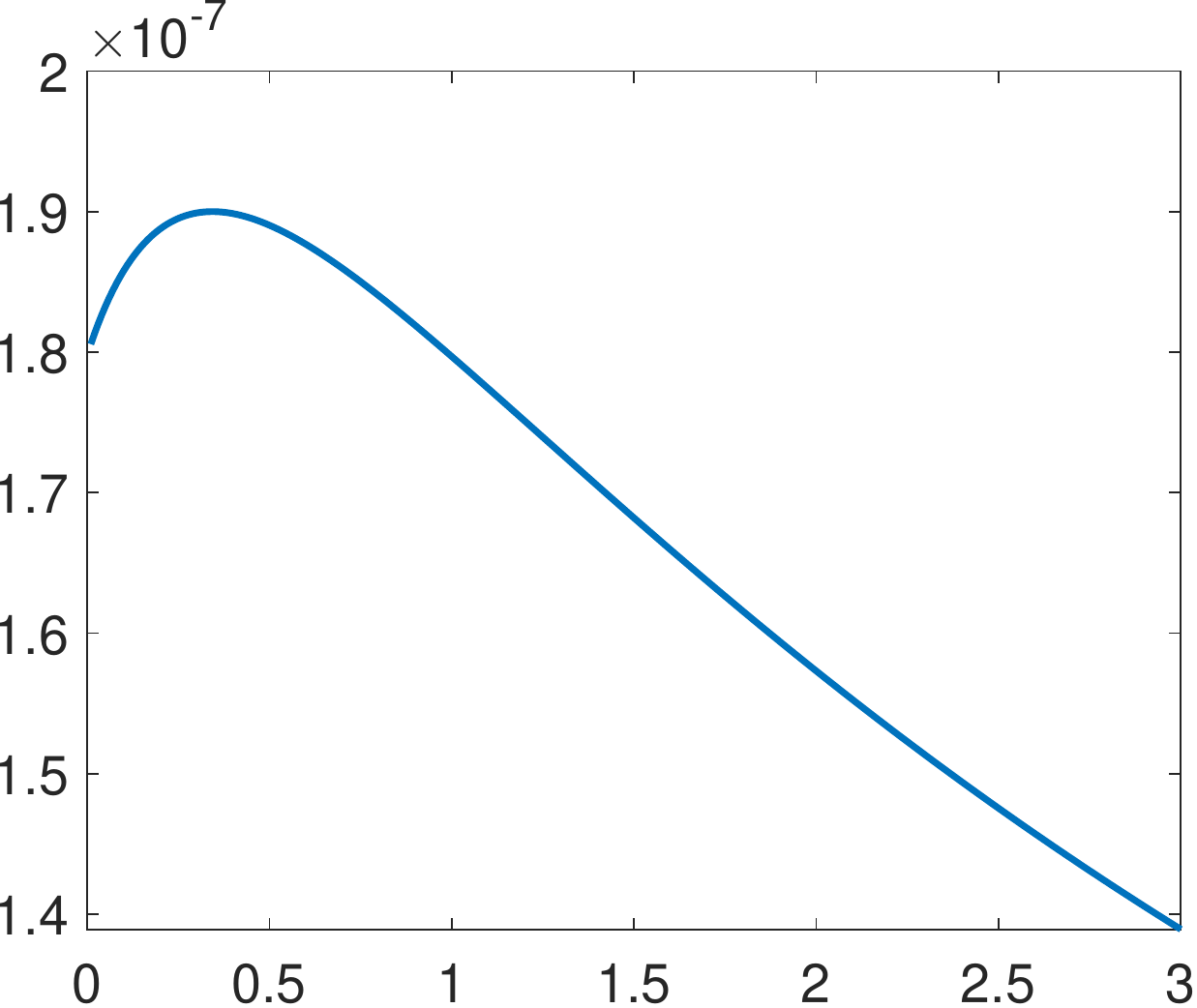} &
  \includegraphics[width=0.28\textwidth]{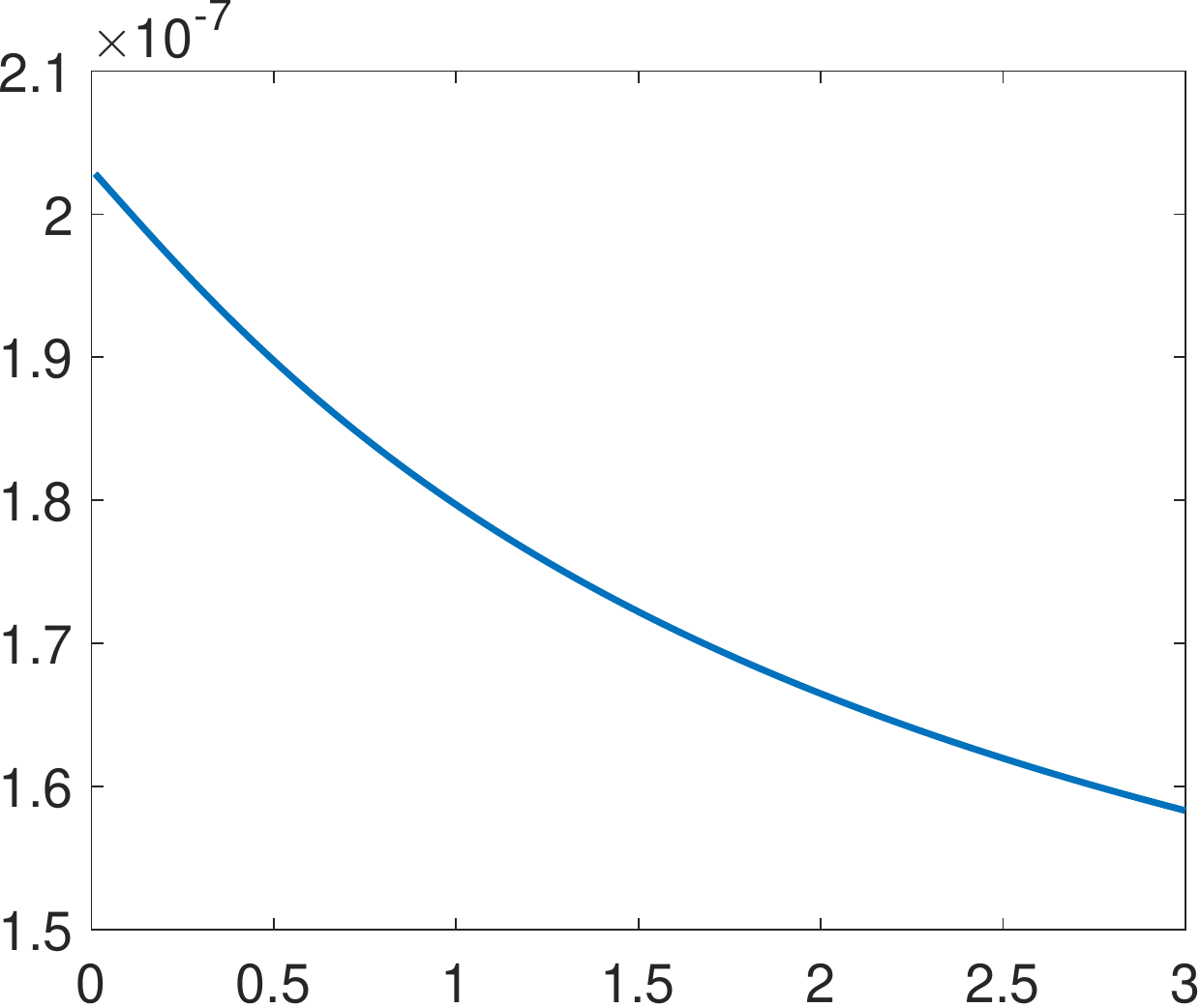}
\end{tabular}
\caption{Single measurement in the top right pixel for a source term in the lower left pixel as a
function of changing the diffusivity in each of the $3\times 3$ pixels.}
\label{fig:non_uniqueness}
\end{figure*}

Intuitively speaking, one can see that rising the diffusivity in the middle pixel increases the measurement since 
particles can easier diffuse through the middle pixel on their way from the lower left to the top right. Rising the diffusivity in the corner pixels
decreases the measurement since particles can easier diffuse to the boundary that is set to zero by the homogeneous Dirichlet condition.
In the middle top, left, bottom, and right pixel, rising the diffusivity first increases the measurement since particles
can easier find their way from the lower left to the top right. But at some point, this effect is reverted since it also drives particles
to the boundary. 

This demonstrates that changing a coefficient can have an increasing or decreasing effect on the measurements, and the effect does not have to be monotonous
for all parameter values. It also indicates that an exact one-dimensional measurement $\hat Y=F_{l,r}(\hat \sigma)$ 
might uniquely determine one parameter in $\hat \sigma$ in some cases (here: the diffusivity in the middle and corner pixels), 
but non-uniqueness might occur in other cases (here: in the middle top, left, bottom, and right pixel).

Let us also stress the following point. A single non-symmetric measurement $F_{l,r}(\sigma)$ with $l\neq r$ might depend non-monotonously on $\sigma$ as demonstrated in 
figure~\ref{fig:non_uniqueness}. But, by lemma~\ref{lemma:FEM_symmetric}, for all $l,r\in H'$, the symmetric measurements $F_{l,l}(\sigma)$, $F_{r,r}(\sigma)$, and also the 
matrix-valued measurement
\[
\begin{pmatrix} F_{l,l}(\sigma) & F_{l,r}(\sigma)\\ F_{r,l}(\sigma) & F_{r,r}(\sigma) \end{pmatrix}\in \R^{2\times 2}
\]
are monotonously non-increasing and convex functions of $\sigma$. Note that the monotonicity and convexity properties of the
matrix-valued measurement hold with respect to the Loewner order even though the individual non-diagonal matrix elements might not have these properties.

\subsection{Non-linearity and local minima}

The inverse problem of recovering $\hat \sigma\in \R^n$ from $F(\hat\sigma)\in \R^m$ could be considered as a non-linear root finding problem and
(for $n=m$) approached with Newton's method which is only known to converge locally. However, in practice one usually 
takes redundant measurements (i.e., $m>n$), and, due to measurement or modelling errors, one cannot expect exact data fit.
Hence, a common approach to reconstruct $\hat \sigma\in \R^n$ from $Y^\delta\approx F(\hat \sigma)\in \R^m$ is to minimize a residuum functional,
e.g.,
\[
R(\sigma):=\norm{F(\sigma)-Y^\delta}^2\to \text{min!}
\]
or a sum of a residuum functional together with a regularization term. 

For our simple $3\times 3$-pixel example, the left image in figure \ref{fig:residuals}  shows a contour plot of $R(\sigma)$ (in a normalized logarithmic scale) for 
a two-dimensional measurement function $F(\sigma):=\begin{pmatrix} F_{\chi_1,\chi_7}(\sigma) & F_{\chi_1,\chi_8}(\sigma)\end{pmatrix}\in \R^2$, exact data 
$Y^\delta=F(\hat \sigma)$,
\[
\hat \sigma=\begin{pmatrix} 1 & 1 & 1 & 0.5 & 1 & 0.5 & 1 &1 &1 \end{pmatrix}^T
\quad \text{ and } \quad
\sigma=\begin{pmatrix} 1 & 1 & 1 & \sigma_4 & 1 & \sigma_6 & 1 &1 &1 \end{pmatrix}^T.
\]
$\sigma_4$ and $\sigma_6$ are varied in steps of $0.002$ up to $0.6$. 
Again, for this choice of unknowns and measurements, we numerically observe non-uniqueness. The results indicate that there is a second point $\tilde \sigma\neq \hat \sigma$ with
$F(\tilde \sigma)=F(\hat \sigma)$. 

\begin{figure*}
\begin{tabular}{c c}
  \includegraphics[height=0.4\textwidth]{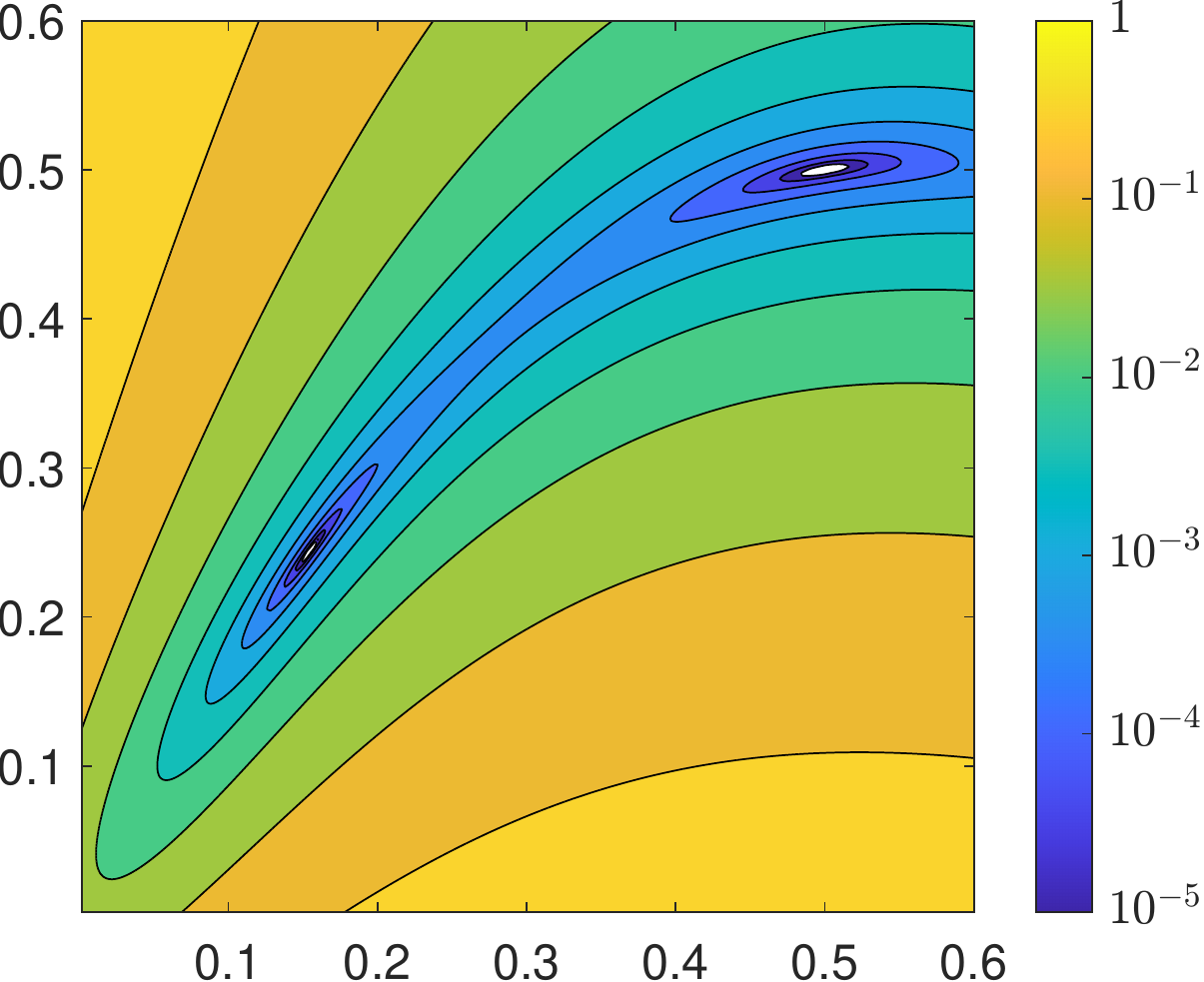} &
  \includegraphics[height=0.4\textwidth]{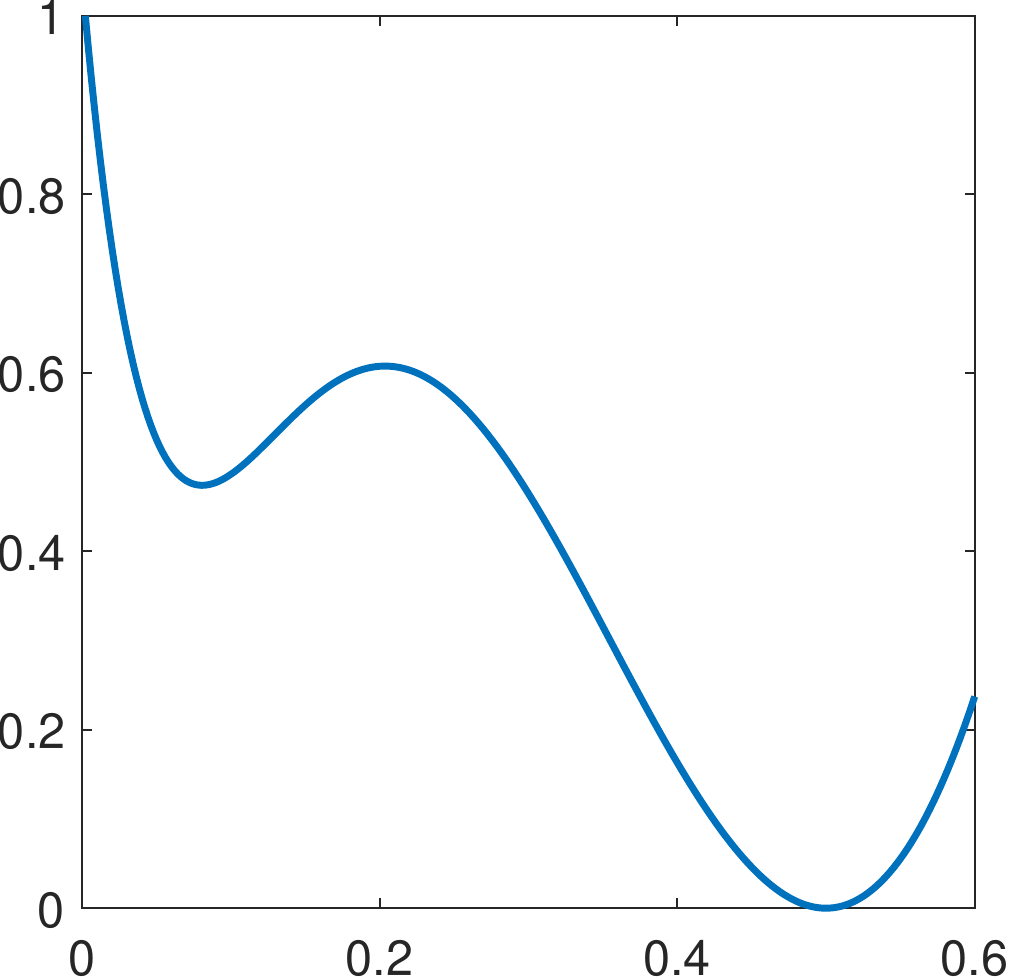}
\end{tabular}
\caption{Residuum functional $R(\sigma)$ as a function of two parameters (left image), and as a function of one parameter (right image).}
\label{fig:residuals}
\end{figure*}

The right image in figure \ref{fig:residuals} shows a plot of $R(\sigma)$ (in a normalized non-logarithmic scale) for varying $\sigma_4$ while keeping $\sigma_6:=\sigma_4$, i.e. the diagonal of the left image in figure \ref{fig:residuals}. It indicates that in this case, one parameter in $\hat \sigma$ can be uniquely reconstructed from the two-dimensional measurement $F(\hat \sigma)$. But it also shows that the residuum functional $R(\sigma)$ possesses a local minimum in a wrong point. Since the curse of dimensionality makes it practically infeasible to numerically find the global minimizer of a non-linear functional in more than a few unknowns, this demonstrates that optimization-based approaches also suffer from only local convergence.

\subsection{Stability, error estimates and ill-posedness}

Even if $F(\hat \sigma)$ uniquely determines $\hat \sigma$, and if this non-linear problem can be solved without running into a local minimum, 
the problem might be ill-posed in the sense that $\sigma$ does not depend stably on $F(\sigma)$. In that case, small errors in the measurements 
might lead to large errors in the reconstruction. The error amplification is often quantified by searching for a Lipschitz stability constant $L>0$ with
\[
\norm{\sigma_1-\sigma_2}\leq L \norm{F(\sigma_1)-F(\sigma_2)} \quad \text{ for all } \sigma_1,\sigma_2\in \R^n.
\]
It should be stressed though, that such a stability estimate does not immediately yield an error estimate for noisy measurements $Y^\delta\approx F(\sigma)$, since 
$Y^\delta$ might not lie in the range of $F$. 

To estimate the (relative) error amplification, we calculate the condition number of $F'(\1)$ for our $3\times 3$ example, where $F:\ \R^9\to \R^{64}$
now depends on all $9$ pixel values, and the components of $F$ are given by $F_{l,r}$ with $l$ and $r$ running through all $8\times 8$-combinations of 
$\chi_{1},\ldots,\chi_{8}$, i.e., we use all combinations of circular subdomain for applying source terms and for measuring the solution. Note that,
lemma~\ref{lemma:FEM_symmetric} implies that roughly half of the measurements are redundant by symmetry, and that unlike in sections \ref{subsect:true_symmetric} and \ref{subsect:FEM_symmetric}, we simply write the measurements as a long vector in order to have $F'(\1)\in \R^{9,64}$ as a matrix.

We repeat the calculation of the condition number of $F'(\1)$ for analogous settings with $n_x\times n_x$-pixels, 
and $(4n_x-4)\times (4n_x-4)$ measurements on circular subdomains in the boundary pixels, cf.\ the left image in figure~\ref{fig:stability}
for the geometry of the $15\times 15$-pixel case. The right image in figure~\ref{fig:stability}
shows the condition number as a function of the total number of pixels $n_x^2$ for $n_x\in \{2,3,\ldots,15\}$.

\begin{figure*}
\begin{tabular}{c c}
  \includegraphics[width=0.45\textwidth]{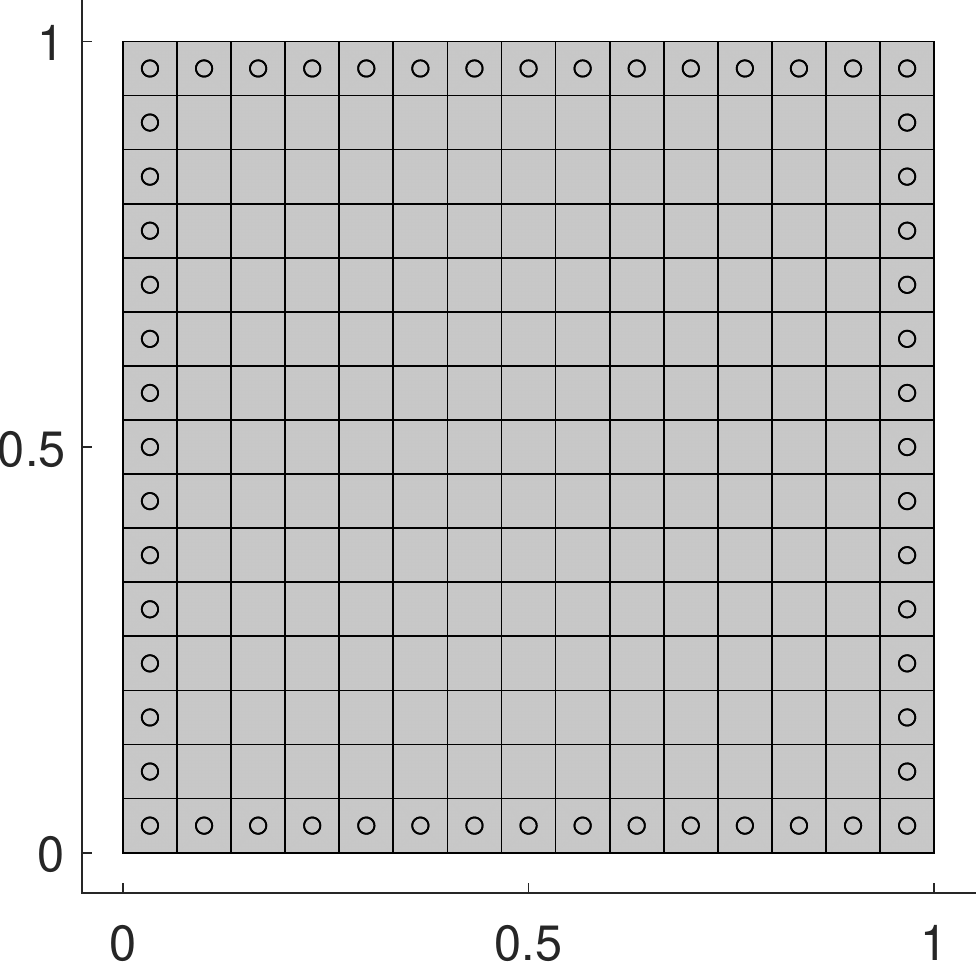} &
  \includegraphics[width=0.45\textwidth]{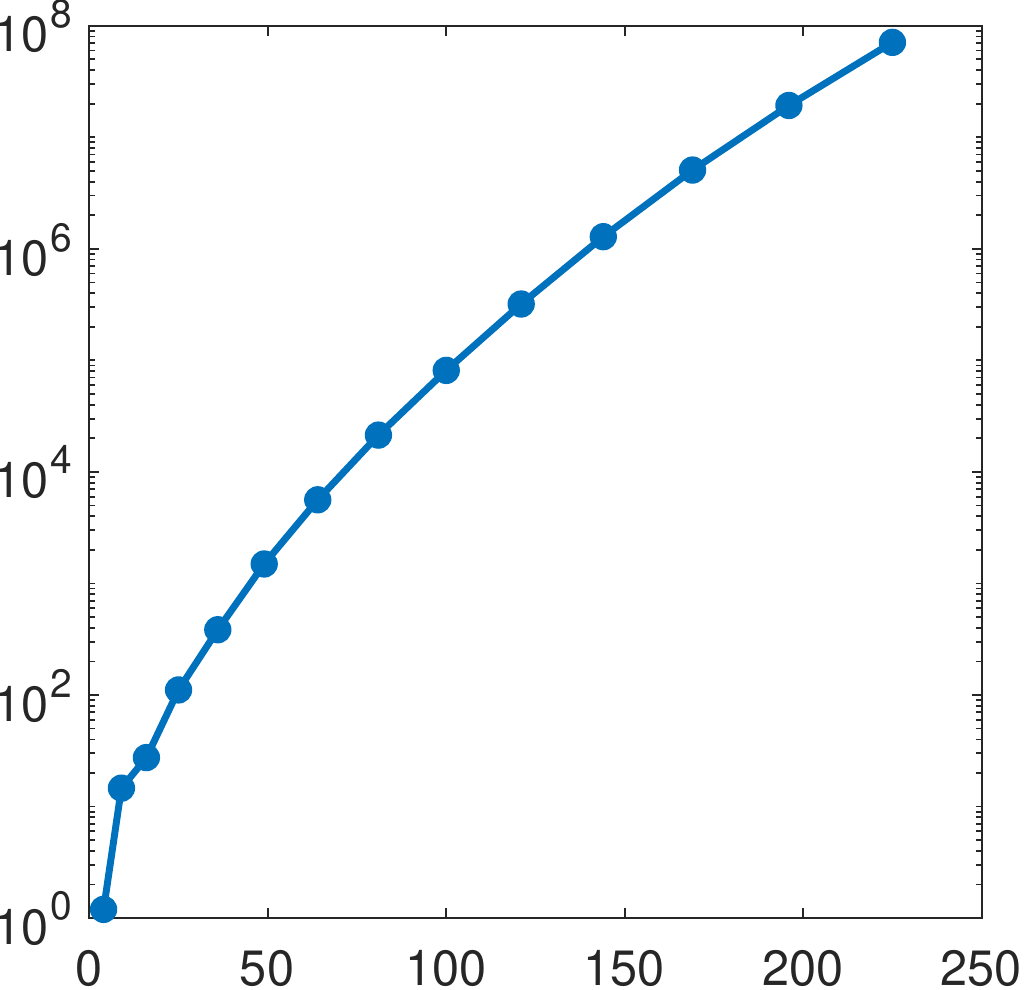}
\end{tabular}
\caption{Condition number of $F'(\1)$ as a function of the total number of pixels (right image) for pixel partitions ranging from $2\times 2$ to $15\times 15$ 
as sketched in left image.}
\label{fig:stability}
\end{figure*}

Our results indicate that the instability of the considered inverse problem grows roughly exponentially with the number of unknowns which
is in par with theoretical results on the similar elliptic inverse coefficient problem of EIT \cite{rondi2006remark}.

\subsection{Further reading}

There is vast literature on theoretical and numerical inverse coefficient problems, their practical applications, and the regularization of ill-posed problems. Let us single out just a very few results as starting points for further reading that are closely connected to the challenges addressed in this section, and the author's own research. Arguably the most prominent inverse elliptic coefficient problem is the so-called Calder\'on problem \cite{calderon1980inverse,calderon2006inverse} with applications in EIT, cf.\ \cite{adler2015electrical,arridge2011methods} for surveys on EIT and the related field of diffuse optical tomography. For theoretical uniqueness proofs in the infinite-dimensional setting of (intuitively speaking) infinitely many pixels and measurements we refer to \cite{uhlmann2009electrical,harrach2009uniqueness,kenig2014recent}. A survey on solving parameter identification problems for PDEs with a focus on sparsity regularization can be found in \cite{jin2012sparsity}. The interplay between instability, regularization and FEM discretization is studied in \cite{kaltenbacher2011adaptive}. A result result on convexification approaches to obtain globally convergent reconstruction algorithms can be found in \cite{klibanov2019convexification}. Learning-based approaches for inverse coefficient problems and parametrized PDEs are studied in \cite{arridge2019solving,seo2019learning,bhattacharya2020model}.

Results on uniqueness and Lipschitz stability for finitely many unknowns from infinitely or finitely many measurements have
been obtained in, e.g., \cite{alessandrini2005lipschitz,alberti2019calderon,harrach2019uniqueness}. 
Let us stress that, for most problems, it is still an open question, how many (and which) measurements are required to uniquely determine
an unknown PDE coefficient with a given resolution, how to explicitly quantify the error amplification, and how to obtain globally convergent reconstruction algorithms.
For a relatively simple, but fully non-linear inverse problem of determining a Robin coefficient in an elliptic PDE, these question 
were recently answered in \cite{harrach2021uniqueness} by exploiting the convexity and monotonicity structure of symmetric measurements from section \ref{subsect:true_symmetric} and \ref{subsect:FEM_symmetric}.



\bibliographystyle{abbrv}
\bibliography{literaturliste}

\begin{thebibliography}{10}

\bibitem{adler2015electrical}
A.~Adler, R.~Gaburro, and W.~Lionheart.
\newblock Electrical impedance tomography.
\newblock In {\em Handbook of Mathematical Methods in Imaging}, pages 701--762.
  Springer, 2015.

\bibitem{alberti2019calderon}
G.~S. Alberti and M.~Santacesaria.
\newblock {C}alder{\'o}n's inverse problem with a finite number of
  measurements.
\newblock In {\em Forum of Mathematics, Sigma}, volume~7. Cambridge University
  Press, 2019.

\bibitem{alessandrini2005lipschitz}
G.~Alessandrini and S.~Vessella.
\newblock Lipschitz stability for the inverse conductivity problem.
\newblock {\em Advances in Applied Mathematics}, 35(2):207--241, 2005.

\bibitem{arridge2019solving}
S.~Arridge, P.~Maass, O.~{\"O}ktem, and C.-B. Sch{\"o}nlieb.
\newblock Solving inverse problems using data-driven models.
\newblock {\em Acta Numerica}, 28:1--174, 2019.

\bibitem{arridge2011methods}
S.~R. Arridge.
\newblock Methods in diffuse optical imaging.
\newblock {\em Philosophical Transactions of the Royal Society A: Mathematical,
  Physical and Engineering Sciences}, 369(1955):4558--4576, 2011.

\bibitem{bhattacharya2020model}
K.~Bhattacharya, B.~Hosseini, N.~B. Kovachki, and A.~M. Stuart.
\newblock Model reduction and neural networks for parametric pdes.
\newblock {\em arXiv preprint arXiv:2005.03180}, 2020.

\bibitem{calderon1980inverse}
A.~P. Calder\'on.
\newblock On an inverse boundary value problem.
\newblock In W.~H. Meyer and M.~A. Raupp, editors, {\em Seminar on Numerical
  Analysis and its Application to Continuum Physics}, pages 65--73. Brasil.
  Math. Soc., Rio de Janeiro, 1980.

\bibitem{calderon2006inverse}
A.~P. Calder\'on.
\newblock On an inverse boundary value problem.
\newblock {\em Comput. Appl. Math.}, 25(2--3):133--138, 2006.

\bibitem{hanke1997regularizing}
M.~Hanke.
\newblock A regularizing {L}evenberg-{M}arquardt scheme, with applications to
  inverse groundwater filtration problems.
\newblock {\em Inverse problems}, 13(1):79, 1997.

\bibitem{harrach2009uniqueness}
B.~Harrach.
\newblock On uniqueness in diffuse optical tomography.
\newblock {\em Inverse Problems}, 25:055010 (14pp), 2009.

\bibitem{harrach2019uniqueness}
B.~Harrach.
\newblock Uniqueness and {L}ipschitz stability in electrical impedance
  tomography with finitely many electrodes.
\newblock {\em Inverse Problems}, 35(2):024005, 2019.

\bibitem{harrach2021uniqueness}
B.~Harrach.
\newblock Uniqueness, stability and global convergence for a discrete inverse
  elliptic {R}obin transmission problem.
\newblock {\em Numerische Mathematik}, 147:29--70, 2021.

\bibitem{jin2012sparsity}
B.~Jin and P.~Maass.
\newblock Sparsity regularization for parameter identification problems.
\newblock {\em Inverse Problems}, 28(12):123001, 2012.

\bibitem{kaltenbacher2011adaptive}
B.~Kaltenbacher, A.~Kirchner, and B.~Vexler.
\newblock Adaptive discretizations for the choice of a tikhonov regularization
  parameter in nonlinear inverse problems.
\newblock {\em Inverse Problems}, 27(12):125008, 2011.

\bibitem{kenig2014recent}
C.~Kenig and M.~Salo.
\newblock Recent progress in the {C}alder{\'o}n problem with partial data.
\newblock {\em Contemp. Math}, 615:193--222, 2014.

\bibitem{klibanov2019convexification}
M.~V. Klibanov, J.~Li, and W.~Zhang.
\newblock Convexification of electrical impedance tomography with restricted
  {D}irichlet-to-{N}eumann map data.
\newblock {\em Inverse Problems}, 2019.

\bibitem{rondi2006remark}
L.~Rondi.
\newblock A remark on a paper by {A}lessandrini and {V}essella.
\newblock {\em Advances in Applied Mathematics}, 36(1):67--69, 2006.

\bibitem{seo2019learning}
J.~K. Seo, K.~C. Kim, A.~Jargal, K.~Lee, and B.~Harrach.
\newblock A learning-based method for solving ill-posed nonlinear inverse
  problems: a simulation study of lung {EIT}.
\newblock {\em SIAM Journal on Imaging Sciences}, 12(3):1275--1295, 2019.

\bibitem{somersalo1992existence}
E.~Somersalo, M.~Cheney, and D.~Isaacson.
\newblock Existence and uniqueness for electrode models for electric current
  computed tomography.
\newblock {\em SIAM J. Appl. Math.}, 52(4):1023--1040, 1992.

\bibitem{uhlmann2009electrical}
G.~Uhlmann.
\newblock Electrical impedance tomography and {C}alder{\'o}n's problem.
\newblock {\em Inverse problems}, 25(12):123011, 2009.

\end{thebibliography}

\newpage

\appendix
\section{Appendix: Source code for COMSOL with MATLAB LiveLink}\label{sect:appendix}

Listing 1--3 contain auxiliary functions to build and manipulate the FEM model. Figure~\ref{fig:diffusion_example}--\ref{fig:stability}
are created by listing 4--9.

\begin{lstlisting}[caption=Build\_model.m,language=MatLab]
function model = Build_Model(nx)
%%%%%%%%%%%%%%%%%%%%%%%%%%%%%%%%%%%%%%%%%%%%%%%%%%%%%%%%%%%%%%%%%%
% Builds the model for the stationary diffusion example with
% nx\times nx pixels and excitations/measurements on circular
% subdomains located in the boundary pixels
%%%%%%%%%%%%%%%%%%%%%%%%%%%%%%%%%%%%%%%%%%%%%%%%%%%%%%%%%%%%%%%%%%

import com.comsol.model.*
import com.comsol.model.util.*

% Clear previously used models and create new one:
ModelUtil.clear; model = ModelUtil.create('Model');

n=nx*nx; model.param.set('n',n);  % total number of pixels 
m=4*nx-4; model.param.set('m',m); % total number of discs

% Create the outer domain Omega as unit square:
cp=model.component.create('cp');
setting=cp.geom.create('setting',2);
Omega=setting.create('Omega','Square'); Omega.set('size',1); 
Omega.set('selresult', true); Omega.set('selresultshow', 'bnd');

% Create pixels, circular subdomains in boundary pixels, 
% and variables to control diffusivity and soure terms 
P_size=1/nx; D_radius=0.15*P_size;
allDnames=cell(1,m); iP=0; iD=0;
for j=1:nx
    y=(j-1)*P_size;
    for k=1:nx
        x=(k-1)*P_size;
        iP=iP+1; Pname=['P',num2str(iP)]; 
        sigmaname=['sigma',num2str(iP)];
        Pi=setting.create(Pname,'Square'); 
        Pi.set('pos', [x,y]); Pi.set('size',P_size);
        Pi.set('selresult', true);
        sigmai=cp.variable.create(sigmaname);
        sigmai.selection.named(['setting_',Pname,'_dom']);
        sigmai.set('sigma', 1); 
        
        if (j==1) || (j==nx) || (k==1) || (k==nx)
            iD=iD+1; Dname=['D',num2str(iD)]; 
            allDnames{iD}=Dname; lname=['l',num2str(iD)];
            Di=setting.create(Dname, 'Circle');
            Di.set('pos', [x+0.5*P_size,y+0.5*P_size]);
            Di.set('r',D_radius); Di.set('selresult', true);
            li=cp.variable.create(lname); 
            li.selection.named(['setting_',Dname,'_dom']); 
            li.set('g', 1);
        end
    end
end

% Complement of union of circular subdomains in Omega:
notD=setting.create('notD', 'Difference');
notD.selection('input').set({'Omega'}); 
notD.selection('input2').set(allDnames); 
notD.set('selresult', true); 
notD.set('keep', true);
l0=cp.variable.create('l0');
l0.selection.named('setting_notD_dom'); l0.set('g', 0); 
setting.run;

% Enter PDE and create FEM-mesh:
PDE=model.physics.create('phys','CoefficientFormPDE','setting');
PDE.feature('cfeq1').set('c','sigma');
PDE.feature('cfeq1').set('f', 'g');
PDE.feature('cfeq1').set('da', {'0'});
bc=PDE.create('boundary_condition','DirichletBoundary',1);
bc.selection.named('setting_Omega_bnd');
FEM_mesh=cp.mesh.create('FEM_mesh');
FEM_mesh.run;

% Prepare solver and solution plot:
study_object=model.study.create('study_object');
study_object.feature.create('stat', 'Stationary');
study_object.run;
plot_object=model.result.create('plot_object', 'PlotGroup2D');
plot_object.set('data', 'dset1');
surf_object=plot_object.create('surf_object', 'Surface');
surf_object.set('expr', 'u');
plot_object.run;



\end{lstlisting}

\bigskip

\begin{lstlisting}[caption=Set\_Variables.m,language=MatLab]
function Set_Variables(model,sigma,l)
%%%%%%%%%%%%%%%%%%%%%%%%%%%%%%%%%%%%%%%%%%%%%%%%%%%%%%%%%%%%%%%%%%
% sets the model variables to given vectors sigma and l
%%%%%%%%%%%%%%%%%%%%%%%%%%%%%%%%%%%%%%%%%%%%%%%%%%%%%%%%%%%%%%%%%%

n=str2num(model.param.get('n'));
m=str2num(model.param.get('m'));

for i=1:n
    sigmai=['sigma',num2str(i)];
    model.component('cp').variable(sigmai).set('sigma', sigma(i));
end

for i=1:m
    li=['l',num2str(i)]; 
    model.component('cp').variable(li).set('g', l(i));
end
\end{lstlisting}\pagebreak

\begin{lstlisting}[caption=Build\_Stiffness\_and\_Load.m,language=MatLab]
function [B_all_1,B,y]=Build_Stiffness_and_Load(model)
%%%%%%%%%%%%%%%%%%%%%%%%%%%%%%%%%%%%%%%%%%%%%%%%%%%%%%%%%%%%%%%%%%
% Calculates stiffness matrix B_all_1 % (with sigma=1 everywhere), 
% B{i}: pixel stiffness matrix of i-th pixel, and
% y: matrix with columnwise load vector y_j 
%    for excitation on j-th circle
%%%%%%%%%%%%%%%%%%%%%%%%%%%%%%%%%%%%%%%%%%%%%%%%%%%%%%%%%%%%%%%%%%

n=str2double(model.param.get('n'));
m=str2double(model.param.get('m'));

l=zeros(1,m); sigma=ones(1,n); Set_Variables(model,sigma,l); 
FEM_struct=mphmatrix(model,'sol1','out','Kc'); 
B_all_1=FEM_struct.Kc;

B=cell(1,n);
y=zeros(size(B_all_1,1),m);
for j=1:n
    sigma=ones(1,n); sigma(j)=2;
    if j<=m
        l=zeros(1,m); l(j)=1;
        Set_Variables(model,sigma,l); 
        FEM_struct=mphmatrix(model,'sol1','out',{'Kc','Lc'});
        y(:,j)=FEM_struct.Lc;
    else
        Set_Variables(model,sigma,l); 
        FEM_struct=mphmatrix(model,'sol1','out',{'Kc'});
    end
    B{j}=FEM_struct.Kc-B_all_1;
end
\end{lstlisting}\pagebreak

\begin{lstlisting}[caption=Plot\_Figure\_Setting.m,language=MatLab]
function Plot_Figure_Setting
%%%%%%%%%%%%%%%%%%%%%%%%%%%%%%%%%%%%%%%%%%%%%%%%%%%%%%%%%%%%%%%%%%
% Creates the two images in figure 1
%%%%%%%%%%%%%%%%%%%%%%%%%%%%%%%%%%%%%%%%%%%%%%%%%%%%%%%%%%%%%%%%%%

model = Build_Model(3);
m=str2double(model.param.get('m'));
n=str2double(model.param.get('n'));

f=figure;
mphviewselection(model,'setting_notD_dom',...
    'edgecolorselected',[0,0,0],'facecolorselected','w',...
    'facecolor','gray'); 
title('');
for j=1:m
     Dname=['D',num2str(j)]; Dlabel=['$D_',num2str(j),'$'];
     Dcenter=str2num(...
         model.geom('setting').feature(Dname).getString('pos'));
     text(Dcenter(1),Dcenter(2),Dlabel,'HorizontalAlignment',...
         'center','FontSize',14,'Interpreter','latex');
     % also remove Dj for the next plot:
     model.geom('setting').feature.remove(Dname);
end
set(gca,'XTick',[0,0.5,1],'YTick',[0,0.5,1],'FontSize',14);
axis([-0.05,1.05,-0.05,1.05])
exportgraphics(f,'fig_pixel_partition_with_circs.eps',...
    'ContentType','vector');

% remove complement of circles to only plot pixel partion
model.geom('setting').feature.remove('notD');

% Plot the pixel partition:
f=figure; mphgeom(model); title(''); 
for i=1:n
    Pname=['P',num2str(i)]; 
    Plabel=['$\mathcal P_',num2str(i),'$'];
    Ppos=str2num(...
        model.geom('setting').feature(Pname).getString('pos'));
    Psize=str2num(...
        model.geom('setting').feature(Pname).getString('size'));
    text(Ppos(1)+Psize/2,Ppos(2)+Psize/2,Plabel,...
        'HorizontalAlignment','center','FontSize',14,...
        'Interpreter','latex');
end
set(gca,'XTick',[0,0.5,1],'YTick',[0,0.5,1],'FontSize',14);
axis([-0.05,1.05,-0.05,1.05])
exportgraphics(f,'fig_pixel_partition.pdf','ContentType','vector');



\end{lstlisting}\pagebreak

\begin{lstlisting}[caption=Plot\_Figure\_Excitations.m,language=MatLab]
function Plot_Figure_Excitations
%%%%%%%%%%%%%%%%%%%%%%%%%%%%%%%%%%%%%%%%%%%%%%%%%%%%%%%%%%%%%%%%%%
% Creates the two images in figure 2
%%%%%%%%%%%%%%%%%%%%%%%%%%%%%%%%%%%%%%%%%%%%%%%%%%%%%%%%%%%%%%%%%%

model = Build_Model(3);
sigma=[1,1,1,1,1,1,1,1,1];
l=[0,1,0,1,1,0,1,0];
Set_Variables(model,sigma,l);
model.study('study_object').run;
model.result('plot_object').run;
f=figure; mphplot(model); title(''); xlabel(''); ylabel('');
set(gca,'XTick',[0,0.5,1],'YTick',[0,0.5,1],'FontSize',14);
axis([-0.05,1.05,-0.05,1.05])
exportgraphics(f,'fig_excitations1.pdf','ContentType','vector');

l=[1,0,1,0,0,1,0,1];
Set_Variables(model,sigma,l);
model.study('study_object').run;
model.result('plot_object').run;
f=figure; mphplot(model); title(''); xlabel(''); ylabel('');
set(gca,'XTick',[0,0.5,1],'YTick',[0,0.5,1],'FontSize',14); 
axis([-0.05,1.05,-0.05,1.05])
exportgraphics(f,'fig_excitations2.pdf','ContentType','vector');


\end{lstlisting}

\bigskip

\begin{lstlisting}[caption=Plot\_Figure\_CompliantMesh.m,language=MatLab]
function Plot_Figure_CompliantMesh
%%%%%%%%%%%%%%%%%%%%%%%%%%%%%%%%%%%%%%%%%%%%%%%%%%%%%%%%%%%%%%%%%%
% Creates the two images in figure 3
%%%%%%%%%%%%%%%%%%%%%%%%%%%%%%%%%%%%%%%%%%%%%%%%%%%%%%%%%%%%%%%%%%

model = Build_Model(3);

% Coarser mesh:
model.component('cp').mesh('FEM_mesh').feature...
    ('size').set('hauto',8);
model.component('cp').mesh('FEM_mesh').run;
f=figure; mphmesh(model); title(''); xlabel(''); ylabel('');
set(gca,'XTick',[0,0.5,1],'YTick',[0,0.5,1],'FontSize',14);
axis([-0.05,1.05,-0.05,1.05])
exportgraphics(f,'fig_compliant_mesh.pdf','ContentType','vector');

% Finer mesh:
model.component('cp').mesh('FEM_mesh').feature...
    ('size').set('hauto',3);
model.component('cp').mesh('FEM_mesh').run;
f=figure; mphmesh(model); title(''); xlabel(''); ylabel('');
set(gca,'XTick',[0,0.5,1],'YTick',[0,0.5,1],'FontSize',14); 
axis([-0.05,1.05,-0.05,1.05])
exportgraphics(f,'fig_compliant_mesh_fine.pdf',...
    'ContentType','vector');

\end{lstlisting}\pagebreak

\begin{lstlisting}[caption=Plot\_Figure\_Non\_uniqueness.m,language=MatLab]
function Plot_Figure_Non_uniqueness
%%%%%%%%%%%%%%%%%%%%%%%%%%%%%%%%%%%%%%%%%%%%%%%%%%%%%%%%%%%%%%%%%%
% Creates the nine images in figure 4
%%%%%%%%%%%%%%%%%%%%%%%%%%%%%%%%%%%%%%%%%%%%%%%%%%%%%%%%%%%%%%%%%%

model = Build_Model(3);
[B_all_1,B,y]=Build_Stiffness_and_Load(model);
y_l=y(:,1); y_r=y(:,end);

sigma=linspace(0.01,3,300);
F=zeros(9,length(sigma));

parfor j=1:length(sigma)
    for i=1:9
        B_sigma=B_all_1+(sigma(j)-1)*B{i};
        lambda_l=B_sigma\y_l; 
        F(i,j)=y_r'*lambda_l;
    end
end

for j=1:9
    figname=['fig_non_uniqueness',num2str(j),'.pdf'];
    f=figure; plot(sigma,F(j,:),'LineWidth',2);
    set(gca,'FontSize',16);
    exportgraphics(f,figname,'ContentType','vector');
end
\end{lstlisting}\pagebreak

\begin{lstlisting}[caption=Plot\_Figure\_Residuals.m,language=MatLab]
function Plot_Figure_Residuals
%%%%%%%%%%%%%%%%%%%%%%%%%%%%%%%%%%%%%%%%%%%%%%%%%%%%%%%%%%%%%%%%%%
% Creates the two images in figure 5
%%%%%%%%%%%%%%%%%%%%%%%%%%%%%%%%%%%%%%%%%%%%%%%%%%%%%%%%%%%%%%%%%%

model = Build_Model(3);
[B_all_1,B,y]=Build_Stiffness_and_Load(model);
y_l=y(:,1); y_r1=y(:,end); y_r2=y(:,end-1); B_4=B{4}; B_6=B{6};

B_sigmahat=B_all_1-0.5*B_4-0.5*B_6; 
lambdahat_l=B_sigmahat\y_l;
Fhat_1=y_r1'*lambdahat_l;
Fhat_2=y_r2'*lambdahat_l;

n=300;
sigma4_range=linspace(0,0.6,n+1); 
sigma4_range=sigma4_range(2:end);
sigma6_range=linspace(0,0.6,n+1);
sigma6_range=sigma6_range(2:end);

[S4range,S6range]=meshgrid(sigma4_range,sigma6_range);
F_1=zeros(n,n); F_2=F_1; 
parfor j=1:n
    for k=1:n
        sigma4=S4range(j,k); sigma6=S6range(j,k);
        B_sigma=B_all_1+(sigma4-1)*B_4+(sigma6-1)*B_6;
        lambda_l=B_sigma\y_l;
        F_1(j,k)=y_r1'*lambda_l;
        F_2(j,k)=y_r2'*lambda_l;
    end
end

R=(F_1-Fhat_1).^2+(F_2-Fhat_2).^2; R=R/max(max(R));
f=figure; contourf(S4range,S6range,log10(R),[-5:0.5:0]);
axis square; 
set(gca,'XTick',0:0.1:0.6,'YTick',0:0.1:0.6,'FontSize',14);
c_handle=colorbar; 
set(c_handle,'TickLabels',{'$10^{-5}$','$10^{-4}$','$10^{-3}$',...
    '$10^{-2}$','$10^{-1}$','$1$'},...
    'TickLabelInterpreter','latex','FontSize',14);
exportgraphics(f,'fig_residuals_2D.pdf','ContentType','vector');


f=figure; 
plot(diag(S4range),diag(R)/max(diag(R)),'LineWidth',2);
set(gca,'FontSize',14); axis square;
exportgraphics(f,'fig_residuals_1D.pdf','ContentType','vector'); 


\end{lstlisting}\pagebreak

\begin{lstlisting}[caption=Plot\_Figure\_Stability.m,language=MatLab]
function Plot_Figure_Stability
%%%%%%%%%%%%%%%%%%%%%%%%%%%%%%%%%%%%%%%%%%%%%%%%%%%%%%%%%%%%%%%%%%
% Creates the two images in figure 6
%%%%%%%%%%%%%%%%%%%%%%%%%%%%%%%%%%%%%%%%%%%%%%%%%%%%%%%%%%%%%%%%%%

nx=2:15;
condnumbers=zeros(size(nx));
for j=1:length(nx)
    model = Build_Model(nx(j));
    [B_all_1,B,y]=Build_Stiffness_and_Load(model);
    lambda=B_all_1\y;
    
    n=str2double(model.param.get('n'));
    m=str2double(model.param.get('m'));
    dF=zeros(m*m,n);
    for i=1:n
        d_iF=-lambda'*B{i}*lambda;
        dF(:,i)=d_iF(:);
    end
    
    condnumbers(j)=cond(dF);
end

f=figure; 
semilogy(nx.^2,condnumbers,'.-','LineWidth',2,'MarkerSize',20);
axis square; set(gca,'FontSize',14);
exportgraphics(f,'fig_stability.pdf','ContentType','vector'); 

% also plot setting with finest pixel partition
f=figure; mphgeom(model); title(''); 
set(gca,'XTick',[0,0.5,1],'YTick',[0,0.5,1],'FontSize',14); 
axis([-0.05,1.05,-0.05,1.05])
exportgraphics(f,'fig_setting_stability.pdf',...
    'ContentType','vector');
\end{lstlisting}


%

\end{document}